\documentclass[parskip=full, twoside=false]{scrartcl} 
\usepackage[utf8]{inputenc} 
\usepackage[T1]{fontenc} 
\usepackage[english]{babel}
\usepackage[autostyle=true]{csquotes}
\usepackage{lmodern} 
\usepackage[backend=biber, style=alphabetic]{biblatex}
\usepackage[pdfauthor={Jonathan Glöckle}, pdftitle={On the space of initial values strictly satisfying the dominant energy condition}]{hyperref} 
\usepackage{titlefoot}
\usepackage{amsmath, amsthm, amssymb, bm, bbm, mathrsfs, mathtools}
\usepackage{extarrows}
\usepackage{tabulary,tabularx}
\usepackage{graphicx}
\usepackage{tikz-cd}
\usepackage[capitalise]{cleveref}
\usepackage{bbold}

\theoremstyle{plain} 
\newtheorem{Satz}{Theorem}[section] 
\newtheorem{Prop}[Satz]{Proposition} 
\newtheorem{Lem}[Satz]{Lemma}
\newtheorem{Kor}[Satz]{Corollary} 
\newtheorem*{Satz*}{Main Theorem} 
\theoremstyle{definition} 
\newtheorem{Def}[Satz]{Definition} 
\newtheorem{Bsp}[Satz]{Example}
\newtheorem{Bem}[Satz]{Remark} 
\newtheorem*{Not}{Notation}

\crefname{Satz}{Theorem}{Theorems}
\crefname{Prop}{Proposition}{Propositions}
\crefname{Lem}{Lemma}{Lemmata}
\crefname{Kor}{Corollary}{Corollaries}
\crefname{Def}{Definition}{Definitions}

\bibliography{refs.bib}

\newcommand{\Z}{\mathbbm{Z}} 
 
\newcommand{\R}{\mathbbm{R}}

\newcommand{\VF}{\mathfrak{X}}
\newcommand{\del}{\partial}

\newcommand{\lto}{\longrightarrow}
\newcommand{\lmapsto}{\longmapsto}
\newcommand{\Frac}[2]{{\textstyle \frac{#1}{#2}}}

\newcommand{\upd}{\mathrm{d}}
\newcommand{\dvol}{\mathrm{dvol}}
\newcommand{\hyp}{\overline{\Sigma}_{Cl}}
\newcommand{\Susp}{S}
\renewcommand{\emptyset}{\varnothing}
\renewcommand{\epsilon}{\varepsilon}

\let\div\relax

\DeclareMathOperator{\Met}{\mathcal{R}}
\DeclareMathOperator{\PSC}{\mathcal{R}^{+}}
\DeclareMathOperator{\Ini}{\mathcal{I}}
\DeclareMathOperator{\DEC}{\mathcal{I}^{+}}

\DeclareMathOperator{\ric}{ric} 
\DeclareMathOperator{\div}{div}
\DeclareMathOperator{\scal}{scal}
\DeclareMathOperator{\grad}{grad}

\DeclareMathOperator{\tr}{tr}
\DeclareMathOperator{\id}{\mathbb{1}}
\DeclareMathOperator{\GF}{C^{\infty}}

\DeclareMathOperator{\Cl}{Cl}
\DeclareMathOperator{\SO}{SO}
\DeclareMathOperator{\Spin}{Spin}

\DeclareMathOperator{\End}{End}
\DeclareMathOperator{\Span}{span}
\DeclareMathOperator{\Fred}{Fred}
\DeclareMathOperator{\ind}{ind}

\DeclareMathOperator{\coker}{coker}

\DeclareMathOperator{\adiff}{\alpha--diff}
\DeclareMathOperator{\oladiff}{\overline{\alpha}--diff}

\renewcommand{\amssubjname}{\textit{2020 Mathematics Subject Classifications. }}

\begin{document}
\title{On the space of initial values strictly satisfying the dominant energy condition}
\author{Jonathan Glöckle\footnote{Jonathan Glöckle, Fakultät für Mathematik, Universität Regensburg, 93040 Regensburg, Germany,
Email address: \href{mailto:jonathan.gloeckle.math@outlook.de}{jonathan.gloeckle.math@outlook.de}}}
\maketitle

\begin{abstract}
	The dominant energy condition imposes a restriction on initial value pairs found on a spacelike hypersurface of a Lorentzian manifold. In this article, we study the space of initial values that satisfy this condition strictly. To this aim, we introduce an index difference for initial value pairs and compare it to its classical counterpart for Riemannian metrics. Recent non-triviality results for the latter will then imply that this space has non-trivial homotopy groups.
\end{abstract}
\amssubjname{53C21, 53C27, 83C05}

\section{Introduction}
\subsection{Dominant energy condition for initial values}
According to general relativity, the universe can be modeled by a time-oriented Lo\-rentz\-ian manifold $(N,\overline{g})$ whose large-scale behavior is governed by the Einstein equation
\begin{align*}
	T=\ric^{\overline{g}}-\frac{1}{2} \scal^{\overline{g}} \overline{g},
\end{align*}
where $T$ denotes the energy-momentum tensor. This does not only apply to the dynamics, the field equations also constraint the physical quantities experienced on a time-slice. More precisely, suppose that $(N,\overline{g})$ contains $M$ as a spacelike hypersurface. On $M$, the induced Riemannian metric $g$ and the second fundamental form $k$, defined with respect to the future-pointing unit normal $e_0$, form a so-called initial value pair $(g,k)$. The Gauß-Codazzi equations imply that it is subject to the Einstein constraints (cf.\ \cite{BI})
\begin{equation} \label{eq:CE}
	\begin{aligned}
		2\rho &= \scal^{g} + (\tr k)^2 - \|k\|^2 \\
		j &= \div k - \upd \tr k,
	\end{aligned} 
\end{equation}
where energy density $\rho = T(e_0,e_0)$ and momentum density $j = T(e_0,-)_{|TM}$ are components of the energy-momentum tensor.

For physical reasons, the energy-momentum tensor is assumed to always satisfy the dominant energy condition, which implies that $\rho \geq \|j\|$. We will say that an initial value pair $(g,k)$ satisfies the dominant energy condition if $\rho \geq \|j\|$, when $\rho$ and $j$ are defined by \eqref{eq:CE}. This condition plays a vital role in the positive mass theorem \cite{SY,Wi} stating that for an asymptotically Euclidean manifold $(M,g)$ with $k$ tending to zero at infinity, the ADM-mass is non-negative if $(g,k)$ satisfies the dominant energy condition.

In this article, we consider the case that $M$ is a compact spin manifold of dimension $n \geq 2$. Our aim is to study the space $\DEC(M)$ of initial value pairs $(g,k)$ for which the dominant energy condition holds strictly, i.e.\ $\rho > \|j\|$ everywhere. This is a subspace of the space $\Ini(M)$ of all initial value pairs, with $C^\infty$-topology. The reason for restricting to the strict version of the dominant energy condition is that it nicely connects to positive scalar curvature, which in turn is rather well-studied.
In \cite{AG}, Ammann and the author discuss some ideas how to extend the results to the (non-strict) dominant energy condition.

\subsection{Connection to positive scalar curvature and main result}
It is a simple observation that if $k \equiv 0$, then the strict dominant energy condition for $(g, k)$ reduces to the condition that $g$ has positive scalar curvature. However, whereas existence of positive scalar curvature metrics imposes a condition on the manifold, this is not true for the strict dominant energy condition. More precisely, we will see later that taking any metric $g$, the pair $(g, \frac{1}{n}\tau g)$ satisfies the dominant energy condition strictly as long as the absolute value of the constant $\tau \in \R$ is large enough. Moreover, such a $\tau$ can be chosen in a way that it continuously depends on the metric $g$ (in $C^2$-topology). This allows to define a comparison map $\Phi \colon \Susp\PSC(M) \simeq \PSC(M) \times [-1,1] \cup \Met(M) \times \{-1,1\} \to \DEC(M)$ by $(g,t) \mapsto (g,\frac{1}{n}\tau(g) t g)$, where $\Met(M)$ is the $C^\infty$-space of metrics, $\PSC(M)$ its subspace of positive scalar curvature metrics and $\Susp$ denotes the suspension.

One of the main approaches to positive scalar curvature is by index theoretic methods. Assume that $(M,g)$ is compact, spin, and of dimension $n$. Then, there is a spinor bundle $\Sigma_{Cl}M$ with a right $Cl_n$-action, called $Cl_n$-linear spinor bundle of $M$. Its Dirac operator $D$ commutes with the $Cl_n$-action and thus gives rise to a $Cl_n$-Fredholm operator, which has a $KO$-valued index called $\alpha$-index $\alpha(M)$. The Schrödinger-Lichnerowicz formula
\begin{align*}
	D^2=\nabla^*\nabla + \frac{1}{4} \scal
\end{align*}
implies that $D$ is invertible if $g$ has positive scalar curvature and so its index vanishes. By homotopy invariance of the index, it is independent of $g$, and so the $\alpha$-index provides an obstruction to existence of positive scalar curvature metrics on $M$ if it is non-zero for some spin structure on $M$.

In the case when there is a positive scalar curvature metric $g_0$ on $M$, this invariant can be refined to a secondary invariant known as index difference that allows to detect non-trivial homotopy groups in the space of positive scalar curvature metrics.
In order to emphasize that it refines the $\alpha$-index and to stress its connection with the $\alpha$-invariant for diffeomorphisms (cf.~\Cite[eq.~(2)]{CSS}), we will call it \emph{$\alpha$-index difference}, or \emph{$\alpha$-difference} for short.
It is constructed as follows:
As before, the $Cl_n$-linear Dirac operator defines a map assigning to each metric a $Cl_n$-Fredholm operator, which is invertible if the metric is of positive scalar curvature. Then applying the $KO$-valued index, we obtain the map
\begin{align*}
	\adiff \colon \pi_k(\PSC(M), g_0) \cong \pi_{k+1}(\Met(M), \PSC(M), g_0) \to KO^{-n-k-1}(\{*\}).
\end{align*} 

A similar invariant exists in the case of initial values. For this, the $Cl_n$-linear spinor bundle has to be replaced by the $Cl_{n,1}$-linear hypersurface spinor bundle $\hyp M$. To define it, we embed $M$ as spacelike hypersurface into a time-oriented spin Lorentzian manifold $(N,\overline{g})$ such that the pair $(g,k)$ arises as induced metric and second fundamental form. Then $\hyp M$ is the restriction of the $Cl_{n,1}$-linear spinor bundle of $N$ to $M$. It turns out that this bundle can be defined intrinsically -- without reference to $N$ -- by $\hyp M = \Sigma_{Cl}M \otimes_{Cl_n} Cl_{n,1}$, i.e.\ it is given by two copies of $\Sigma_{Cl}M$. The role of the Dirac operator is now played by the Dirac-Witten operator $\overline{D}$, which is $Cl_{n,1}$-linear in our case, and which will be defined in \cref{sec:DW-Op} below. There is a Schrödinger-Lichnerowicz type formula for $\overline{D}$
\begin{align*}
	\overline{D}^2 = \overline{\nabla}^*\overline{\nabla} +\frac{1}{2}(\rho-e_0 \cdot j^\sharp \cdot),
\end{align*}
which ensures that $\overline{D}$ is invertible if $(g,k)$ strictly satisfies the dominant energy condition. With these changes, the same construction as before yields an index difference for initial values
\begin{align*}
	\oladiff \colon \pi_k(\DEC(M), (g_0, k_0)) \cong \pi_{k+1}(\Ini(M), \DEC(M), (g_0, k_0)) \to KO^{-n-k}(\{*\}),
\end{align*}
where $(g_0, k_0) \in \DEC(M)$.
Notice that there is a degree shift in the target compared to $\adiff$:
This results from the $\Cl_{n,1}$-linearity of the Dirac-Witten operator in contrast to the $\Cl_n$-linearity of the Dirac operator.

\begin{Not}
	To avoid clumsy notation, we often write $\adiff(g_{-1},g_1)$ for the $\alpha$-difference applied to the $\pi_0$-class represented by $(S^0, 1) \to (\PSC(M), g_1),\, t \mapsto g_t$.
	Likewise, we write $\oladiff((g_{-1},k_{-1}),(g_1,k_1))$ for the $\overline{\alpha}$-difference of the $\pi_0$-class defined by $(S^0, 1) \to (\DEC(M), (g_1,k_1)),\, t \mapsto (g_t,k_t)$.
\end{Not}

Unlike the situation of the $\alpha$-difference, where the $\alpha$-index constitutes an interesting invariant obstructing positive scalar curvature, there is no interesting primary invariant associated with the $\overline{\alpha}$-difference:
The index of the Dirac-Witten operator $\overline{D}$ is always zero.
This follows for example from the observation that the dominant energy condition is not obstructed, since, as mentioned above, $(g, \frac{1}{n}\tau g) \in \DEC(M)$ for $g \in \Met(M)$ and suitably large $\tau \in \R$.
The $\overline{\alpha}$-difference, however, is an interesting invariant.
This is a consequence of the main theorem of this paper, where we compare it to the $\alpha$-difference or, in the case of the $\pi_0$-part, to the $\alpha$-index.
\begin{Satz*}
\begin{enumerate}
	\item For $g_0 \in \PSC(M)$ and all $k \geq 0$, the diagram
	\begin{equation*}
	\begin{tikzcd}
	\pi_k(\PSC(M), g_0)  \ar[dr, "\adiff"'] \rar{\mathrm{Susp}} & \pi_{k+1}(\Susp\PSC(M), [g_0,0]) \rar{\Phi_*}&
	\pi_{k+1}(\DEC(M), (g_0,0)) \ar{dl}{\oladiff} \\ 
	&  KO^{-n-k-1}(\{*\}) &
	\end{tikzcd}
	\end{equation*}
	commutes.
	\item	For $g_0 \in \Met(M)$,
	\begin{align*}
	\oladiff\left(\left(g_0,-\frac{1}{n}\tau(g_0)g_0 \right), \left(g_0,\frac{1}{n}\tau(g_0)g_0 \right)\right) = \alpha(M) \in KO^{-n}(\{*\}).
	\end{align*}
\end{enumerate}
\end{Satz*}
The idea of the proof is the following: For a pair of the form $(g, \frac{1}{n}\tau(g)tg)$, $t \in \R$, the $Cl_{n+1}$-linear Dirac-Witten operator is given by $\overline{D}=D \otimes_{Cl_n} Cl_{n,1} - \tau(g) t L(e_0)$, where $D$ is the $Cl_{n}$-linear Dirac operator from before and $L(e_0)$ is left multiplication with the future-pointing unit normal on $M$ when $M$ is considered as spacelike hypersurface of $N$ as above. Now, we observe that the $Cl_{n,1}$-structure of $\hyp M$ given by right multiplication can be extended to a $Cl_{n+1,1}$-structure by setting the right multiplication by the additional basis vector as $\tilde R(e_{n+1}) \coloneqq L(e_0)a$, where $a$ is the even-odd grading operator. With this $Cl_{n+1,1}$-structure, $\hyp M$ corresponds to $\Sigma_{Cl} M$ under the Morita equivalence relating $Cl_n$- and $Cl_{n+1,1}$-modules. Moreover, under this equivalence $D \otimes_{Cl_n} Cl_{n,1}$ is associated to $D$ and, by definition, the index map is invariant under this correspondence. The second summand can be understood as coming from the Bott map, which assigns to a $Cl_{n+1,1}$-Fredholm operator $F$ the family of $Cl_{n,1}$-Fredholm operators $[-1,1] \ni t \mapsto F+t\tilde{R}(e_{n+1}) a = F +tL(e_0)$. Again, invariance of the index map under this assignment is a consequence of its definition, but an extra sign has to be taken into account resulting from the fact that in the definition of the index map Morita equivalence and Bott map are applied in the reverse order.

As a consequence of the main theorem, every element in $\pi_k(\PSC(M), g_0)$ with non-trivial $\alpha$-difference gives rise to a non-zero element in $\pi_{k+1}(\DEC(M), (g_0,0))$. Such elements have been constructed for example by Hitchin \cite{Hi}, Hanke, Schick and Steimle \cite{HSS}, Botvinnik, Ebert and Randal-Williams \cite{BER} as well as Crowley, Schick and Steimle \cite{CSS} using different techniques. In particular, we obtain the following corollary.
\begin{Kor} \label{Kor:Main}
	\begin{enumerate}
		\item If $M$ is a compact spin manifold of dimension $n \geq 6$ that admits a metric of positive scalar curvature, then $\DEC(M)$ is not contractible.
		\item If $M$ is a compact spin manifold of dimension $n \geq 2$ with $\alpha(M) \neq 0$ (in particular, $M$ does not carry a positive scalar curvature metric), then $\DEC(M)$ is not connected.
	\end{enumerate}
\end{Kor}

The structure of the article is as follows. In the first chapter, we review the $KO$-valued index map and the construction of the $\alpha$-difference. Much of this material is owed to Ebert \cite{Eb}. The second chapter is devoted to the construction of the $\overline{\alpha}$-difference. To this end, the $Cl_{n,1}$-linear hypersurface spinor bundle and its Dirac-Witten operator are introduced. We discuss the $Cl_{n,1}$-linear version of the Dirac-Witten operator in some detail, as it seems not to have been studied before. In the last chapter, we construct the comparison map, prove the main theorem and discuss some more of its consequences. 

\subsection*{Acknowledgements}
I would like to thank Bernd Ammann for the idea for this project as well as his support during its execution.
I would also like to thank the anonymous referee, who went through the manuscript very  carefully and helped improving it at many spots. 
I was supported by the SFB 1085 “Higher Invariants” funded by the Deutsche Forschungsgemeinschaft.

\section{The classical \texorpdfstring{$\alpha$}{alpha}-index difference}
\subsection{\texorpdfstring{$KO$}{KO}-theory via Fredholm operators}
This section is devoted to the $KO$-valued index map, a map that associates to a family of Clifford-linear Fredholm operators an element in $KO$-theory. In its description, we will stick closely to the framework presented in Ebert \cite{Eb} that we briefly recall. All Hilbert spaces are understood as being real and separable. A \emph{$Cl_{n,k}$-Hilbert space} $H$ is always $\Z/2\Z$-graded. Typically, the $\Z/2\Z$-grading is given in terms of a grading operator $\iota \colon H \to H$, and the Clifford action is determined by a Clifford multiplication $c \colon \R^{n,k} \to \End(H)$, where $\R^{n,k}$ is the pseudo-Euclidean vector space $\R^n \oplus \R^k$ with the standard inner product that is positive definite on the first summand and negative definite on the second one. The convention for the Clifford multiplication is such that $c(v)c(w) + c(w)c(v) = -2\langle v, w \rangle$.

If $(H,\iota,c)$ is a $Cl_{n,k}$-Hilbert space, then $c$ gives rise to a representation $Cl_{n,k} \to \End(H)$, which can be decomposed into irreducible ones. $(H,\iota,c)$  is called \emph{ample}, if it contains each irreducible representation infinitely often. By the structure theory for real Clifford representations, this just means that $H$ is infinite-dimensional if $n-k \not\equiv 0 \mod 4$, and amounts to the condition that both the $+1$- and the $-1$-eigenspace of the volume element $\omega_{n,k}\coloneqq \iota c(e_1) \cdots c(e_{n+k})$ are infinite-dimensional if $n-k \equiv 0 \mod 4$.

\begin{Def}
	Let $(H, \iota,c)$ be an ample $Cl_{n,k}$-Hilbert space. Then a \emph{$Cl_{n,k}$-Fredholm operator} $F$ is a (bounded) Fredholm operator on $H$ that is self-adjoint, odd with respect to $\iota$,  $Cl_{n,k}$-linear and, in the case $n-k \equiv -1 \mod 4$, satisfies the additional condition that $\omega_{n,k}F\iota$ is neither essentially positive nor essentially negative. We denote by $\Fred^{n,k}(H)$ the space of $Cl_{n,k}$-Fredholm operators with operator norm topology. Furthermore, we write $G^{n,k}(H) \subseteq \Fred^{n,k}(H)$ for the subspace of invertible elements. 
\end{Def}

Note that we have $\Fred^{n+1,k}(H) \subseteq \Fred^{n,k}(H)$ and $\Fred^{n,k+1}(H) \subseteq \Fred^{n,k}(H)$:
In the cases $n-k =1,2 \mod 4$, this is immediate.
If $n-k =0 \mod 4$, this follows since the additional generator of the extended Clifford action on the $\Cl_{n,k}$-Hilbert space $H$ anti-commutes with $\omega_{n,k}$.
Finally, in case $n-k \equiv -1 \mod 4$, we use that for a $\Cl_{n+1,k}$- or $\Cl_{n,k+1}$-linear operator $F$, the additional generator of the extended Clifford action anti-commutes with $\omega_{n,k}F\iota$.

\begin{Bem}
	As was pointed out by the referee, ampleness of $H$ and the additional condition in the case where $n-k \equiv -1 \mod 4$ are only needed to ensure bijectivity of the index map discussed below and are not necessary for its existence.
	For instance, the inductive extension of the index map from degree $n-1$ to degree $n$ does not require the left hand vertical map in diagram~\eqref{eq:indbott} to be an isomorphism.
	Since in this article we will not use that the index map is an isomorphism, all discussions about ampleness and the additional condition are included for the sake of completeness only (and shifted to a large extent to the appendix).
\end{Bem}

\begin{Bsp} \label{Bsp}
	The archetypical example of a $Cl_{n,0}$-Fredholm operator is (the bounded transform of) the $Cl_{n}$-linear Dirac operator on a closed Riemannian spin manifold $(M,g)$ of dimension $n>0$: Let $P_{\Spin(n)}M \to P_{SO(n)}M$ be a spin structure of $M$. The \emph{$Cl_{n}$-linear spinor bundle} is $\Sigma_{Cl}M \coloneqq P_{\Spin(n)}M \times_\ell Cl_{n}$, where $\ell \colon \Spin(n) \to \End(Cl_{n})$ is given by left multiplication. Its name derives from the fact that right multiplication in $Cl_n$ induces a \emph{right Clifford multiplication} $R \colon \R^n \to \End(\Sigma_{Cl}M)$, which commutes with the left Clifford multiplication by tangent vectors. Furthermore, it carries a $\Z/2\Z$-grading $a$ induced by $Cl_n \to Cl_n, \; \R^n \ni v \mapsto -v$, the \emph{even-odd-grading}. The bundle metric induced by the metric on $Cl_n$ that makes the standard basis $(e_{i_1} \ldots e_{i_l})_{0 \leq l \leq n+k, 1 \leq i_1 < \dots < i_l \leq n+k}$ orthonormal allows to define an $L^2$-scalar product and the space of $L^2$-sections $H \coloneqq L^2(M,\Sigma_{Cl}M)$. Both $a$ and $R$ descend to $H$, turning $(H,a,R)$ into an ample $Cl_n$-Hilbert space. The \emph{$Cl_n$-linear Dirac operator} $D$, i.e.\ the Dirac operator of $\Sigma_{Cl}M$ w.r.t.\ the connection induced by the Levi-Civita connection, can be viewed as unbounded operator on $H$. By standard results on the analysis of Dirac operators, its bounded transform $F\coloneqq \frac{D}{\sqrt{1+D^2}}$ is a Fredholm operator on $H$, and as $D$ is $Cl_n$-linear (w.r.t.\ $R$) and odd (w.r.t.\ $a$), so is $F$. Thus, $F \in \Fred^{n,0}(H)$, whereby the additional condition for $n \equiv -1 \mod 4$ is well-known to be satisfied for Dirac type operators. In order to be self-contained, we recall this in the appendix. It is worth noting that the Schrödinger-Lichnerowicz formula implies that $F$ is invertible, so $F \in G^{n,0}(H)$, if $g$ is a metric of positive scalar curvature.
\end{Bsp}

The following consequence of Kuiper's theorem is proven in \cite{Eb}. It is one of the main ingredients for translating the classical results from \cite{AS} into the present framework.
\begin{Prop} \label{Prop:Cntr}
	The space $G^{n,k}(H)$ is contractible for all $n,k \geq 0$.
\end{Prop}

\begin{Satz}[Index map] \label{Thm:KOviaCl}
	If $H$ is an ample $Cl_{n,k}$-Hilbert space, then $\Fred^{n,k}(H)$ represents $KO$-theory: For compact relative CW-complexes $(X,Y)$, there is a natural (in $(X,Y)$) bijection
	\begin{align*}
	\ind \colon	[(X,Y),(\Fred^{n,k}(H),G^{n,k}(H))] \lto KO^{k-n}(X,Y)
	\end{align*}
	called \emph{index map}. Moreover, $\ind$ is invariant under $Cl_{n,k}$-Hilbert space isomorphisms, i.e.\ if $U \colon H \to H^\prime$ is an isomorphism of $Cl_{n,k}$-Hilbert spaces, then
	\begin{equation*}
	\begin{tikzcd}[column sep=tiny]
	{[(X,Y),(\Fred^{n,k}(H),G^{n,k}(H))]} \ar[rr,"\cong"] \ar[dr,"\ind",swap] & & {[(X,Y),(\Fred^{n,k}(H^\prime),G^{n,k}(H^\prime))]}  \ar[dl,"\ind"] \\
	&  KO^{k-n}(X,Y) &
	\end{tikzcd}
	\end{equation*}
	commutes, where the upper map is induced by $\Fred^{n,k}(H) \ni F \mapsto UFU^{-1}$.
\end{Satz}

The index map is constructed inductively, the starting point being the index of a family of $Cl_{0,0}$-Fredholm operators, i.e.\ odd Fredholm operators on a $\Z/2\Z$-graded Hilbert space. Here, the corresponding statement is known as Atiyah-Jänich theorem (cf.\ \cite[Thm.\ 2.17]{MA} for a detailed derivation from the version in \cite{AS}).

The generalization to arbitrary $n$ (but still with $k=0$) is provided by the Bott map.
\begin{Satz}[Bott map, {\cite[Thm. A(k)]{AS}}] \label{Thm:KOviaCl2}
	For compact CW-pairs $(X,Y)$,  the map
	\begin{align*} \nonumber
	[(X,Y),(\Fred^{n+1,k}(H),G^{n+1,k}(H))] &\lto [(X,Y)\times (I,\del I),(\Fred^{n,k}(H),G^{n,k}(H))]
	\\
	[x \mapsto F_x] &\lmapsto [(x,t) \mapsto F_x+tc(e)\iota]
	\end{align*}
	is a natural bijection.\footnote{For two pairs $(X,A)$ and $(Y,B)$, we write $(X,A) \times (Y,B) \coloneqq (X \times Y, X \times B \cup A \times Y)$.} Here, $e$ is the additional basis vector of $\R^{n+1,k}$ compared to $\R^{n,k}$ and $I=[-1,1]$.
\end{Satz}
As $(X \times I)/(Y  \times I \cup X \times \del I) \cong \Sigma_{red} X/Y$ the right hand isomorphism in the following diagram exists, and the defintion of the index map can be extended inductively by requiring that it commutes:
\begin{equation} \label{eq:indbott}
\begin{tikzcd}[column sep=small]
{[(X,Y), (\Fred^{n,0}(H), G^{n,0}(H))]} \ar[r,dashed, "\ind"] \dar{\cong}& 
KO^{-n}(X,Y) \ar["\cong"]{d} \\
{[(X,Y)\times (I, \del I), (\Fred^{n-1,0}(H), G^{n-1,0}(H))]}  \rar{\ind} & 
KO^{-n+1}(X \times I \,,\,X \times \del I \cup Y \times I).
\end{tikzcd}
\end{equation}

The extension to arbitrary $k$ uses periodicity statements in the theory of $Cl_{n,k}$-Hilbert spaces known as Morita equivalences. One of them states that the categories of $Cl_{n,k}$-Hilbert spaces and $Cl_{n+1,k+1}$-Hilbert spaces are equivalent. Its construction is the following: A $Cl_{n,k}$-Hilbert space $(H,\iota,c)$ defines a $Cl_{n+1,k+1}$-Hilbert space structure on $H \oplus H$ by 
\begin{align} \label{eq:mor1} \nonumber
\tilde{\iota} &= 
\begin{pmatrix}
\iota & \phantom{-}0 \\
0 & -\iota
\end{pmatrix}\\ \nonumber
\tilde{c}(v)&= 
\begin{pmatrix}
c(v) & 0 \\
0 & -c(v)
\end{pmatrix} \quad \text{for all} \; v \in R^{n+k} \oplus 0\\
\tilde{c}(e)&= 
\begin{pmatrix}
0 & -\id \\
\id & \phantom{-}0
\end{pmatrix} \\ \nonumber
\tilde{c}(\epsilon)&= 
\begin{pmatrix}
0 & \id \\
\id & 0
\end{pmatrix},
\end{align}
where we view $\R^{n+1,k+1}$ as $R^{n,k} \oplus \R e \oplus \R \epsilon$. And a morphism $F \colon H \to H^\prime$ of $Cl_{n,k}$-Hilbert spaces gives rise to a morphism 
\begin{align*} \tilde F= \begin{pmatrix} F & 0 \\ 0 & F \end{pmatrix} \colon H \oplus H \to H^\prime \oplus H^\prime \end{align*}
of the corresponding $Cl_{n+1,k+1}$-Hilbert spaces. Conversely, for a $Cl_{n+1,k+1}$-Hilbert space $(H, \iota, c)$, the restrictions of the structure maps to $H_0 \coloneqq \ker(c(\epsilon)c(e)-\id)$ yield a $Cl_{n,k}$-Hilbert space, and morphisms of  $Cl_{n+1,k+1}$-Hilbert spaces restrict to morphisms of these $Cl_{n,k}$-Hilbert spaces. These constructions are seen to be mutually inverse up to natural isomorphism.

Another Morita equivalence exists between $Cl_{n+4,k}$-Hilbert spaces and $Cl_{n,k+4}$-Hilbert spaces. For this, we regard both $\R^{n+4,k}$ and $\R^{n,k+4}$ as $\R^n \oplus \R^k \oplus \Span \{e_1,e_2,e_3,e_4\}$, where $e_1, \ldots e_4$ are the last four basis vectors of $\R^{n+4}$ or the last four basis vectors of $\R^{k+4}$, respectively. Given a $Cl_{n+4,k}$-Hilbert space $(H,\iota,c)$, we can define a $Cl_{n,k+4}$-Hilbert space $(H, \iota, \tilde{c})$ by $\tilde{c}_{|\R^{n,k}}=c_{|\R^{n,k}}$ and $\tilde{c}(e_i) = \eta c(e_i)$ for $\eta=c(e_1)\cdots c(e_4)$. Morphisms are mapped to the morphisms defined by the same underlying bounded linear maps. The inverse procedure is given similarly, by assigning to a $Cl_{n,k+4}$-Hilbert space $(H,\iota,\tilde{c})$ the $Cl_{n+4,k}$-Hilbert space $(H,\iota,c)$ with $c_{|\R^{n,k}}=\tilde{c}_{|\R^{n,k}}$ and $c(e_i) = \tilde\eta \tilde{c}(e_i)$, where $\tilde\eta=\tilde{c}(e_1)\cdots \tilde{c}(e_4)$.

These equivalences are accompanied by homeomorphisms between the spaces of Clifford-linear Fredholm operators.
\begin{Prop} \label{Prop:MorMap}
	The Morita equivalences discussed above induce homeomorphisms of pairs
	\begin{align*}
	(\Fred^{n,k}(H),G^{n,k}(H)) &\lto (\Fred^{n+1,k+1}(H \oplus H),G^{n+1,k+1}(H \oplus H)) \\
	F &\lmapsto \begin{pmatrix} F & 0 \\ 0 & F \end{pmatrix} \\
	\intertext{and}
	(\Fred^{n+4,k}(H),G^{n+4,k}(H)) &\lto (\Fred^{n,k+4}(H),G^{n,k+4}(H)) \\
	F &\lmapsto F.
	\intertext{In particular, there is a homeomorphism}
	(\Fred^{n,k}(H),G^{n,k}(H)) &\lto (\Fred^{n+8,k}(H \otimes \R^{16}),G^{n+8,k}(H \otimes \R^{16})) \\
	F &\lmapsto F \otimes \id_{\R^{16}}.
	\end{align*}
\end{Prop}

The index map is then defined inductively for all $(n,k)$ with $0 \leq k \leq n$ by the requirement that
\begin{equation} \label{eq:indmor}
\begin{tikzcd}	
{[(X,Y), (\Fred^{n,k}(H), G^{n,k}(H))]} \ar[r,dashed, "\ind"] & 
KO^{k-n}(X,Y) \ar[equal]{d} \\
{[(X,Y), (\Fred^{n-1,k-1}(H_0), G^{n-1,k-1}(H_0))]} \uar{\cong} \rar{\ind} & 
KO^{k-n}(X,Y)
\end{tikzcd}
\end{equation}
commutes. Lastly, it is extended to the missing $(n,k)$ with $0 \leq n,k$ by commutativity of
\begin{equation} \label{Diag:Bott8}
\begin{tikzcd}	
{[(X,Y), (\Fred^{n,k}(H), G^{n,k}(H))]} \dar{\cong} \ar[r,dashed,"\ind"] & 
KO^{k-n}(X,Y) \dar{\cdot x} \\
{[(X,Y), (\Fred^{n+8,k}(H \otimes \R^{16}), G^{n+8,k}(H \otimes \R^{16}))]} \rar{\ind} & 
KO^{k-n-8}(X,Y),
\end{tikzcd}
\end{equation}
where $x$ denotes a generator of $KO^{-8}(\{*\})$.

\begin{Bem}
	The commutativity of \eqref{Diag:Bott8} does not only hold for $n<k$ (where it is true by definition), but is also true for $k \leq n$ provided that the right generator $x \in KO^{-8}(\{*\})$ is chosen. This follows from the last remark in \cite{AS}.
\end{Bem}

\begin{Bsp}
	In the setting of \cref{Bsp}, we can define the $\alpha$-index of $M$ by $\alpha(M) = \ind(F) \in KO^{-n}(\{*\})$. This invariant was first defined by Hitchin \cite{Hi} and is a well-known obstruction to positive scalar curvature: From the continuity of the assignment $g \mapsto F_g$ discussed in the next section, it follows that $\alpha(M)$ is independent of the metric on $M$ (in fact, it is even spin-bordism invariant) and so has to vanish for every spin structure if $M$ carries a positive scalar curvature metric.
\end{Bsp}

\subsection{Construction of the \texorpdfstring{$\alpha$}{alpha}-index difference} \label{sec:alpha}
Let $M$ be a compact spin manifold of dimension $n>0$ that has a positive scalar curvature metric $g_0$. The $\alpha$-index difference, also introduced by Hitchin \cite{Hi}, is a family version of the $\alpha$-index. More precisely, $\adiff \colon \pi_k(\PSC(M),g_0) \to KO^{-n-k-1}(\{*\})$ arises in the following way: As $\Met(M)$ is contractible, the long exact sequence for homotopy groups implies $\pi_k(\PSC(M),g_0) \cong \pi_{k+1}(\Met(M),\PSC(M),g_0)$. For each metric $g$, the $Cl_n$-linear Dirac operator $D_g$ defines a $Cl_n$-linear Fredholm operator 
\begin{align*}
F_g=\frac{D_g}{\sqrt{1+D_g^2}},
\end{align*}
which is invertible if $g \in \PSC(M)$. The assignment $g \mapsto F_g$ gives rise to a map $(\Met(M),\PSC(M)) \to (\Fred^{n,0},G^{n,0})$, which induces a map to $\pi_{k+1}(\Fred^{n,0},G^{n,0},F_{g_0})$. Applying the index map from the last section, we obtain an element in $KO^{-n}(D^{k+1},S^k) \cong KO^{-n-k-1}(\{*\})$.

In this outline, however, we glossed over the detail that the $Cl_n$-linear spinor bundles and hence the $L^2$-spaces, on which the Fredholm operators $F_g$ act, depend on the metric $g$. These $L^2$-spaces form a Hilbert bundle over $\Met(M)$, which, by Kuiper's theorem, can be trivialized. Such a trivialization allows to define the map $(\Met(M),\PSC(M)) \to (\Fred^{n,0},G^{n,0})$. We will make this more explicit: The $Cl_n$-linear spinor bundles for different metrics can be identified using the method of generalized cylinders due to Bär, Gauduchon and Moroianu \cite{BGM}. This gives rise to a specific trivialization of the Hilbert bundle of $L^2$-spaces.

Let us start with this construction by fixing a topological spin structure on $M$, i.e.\ a double covering 
\begin{align*}
P_{\widetilde{GL}^+(n)}M \to P_{GL^+(n)}M
\end{align*}
over the principal bundle of positively oriented frames of $TM$. This defines, for any $g \in \Met(M)$, a spin structure for $(M,g)$ by pullback
\begin{equation*}
\begin{tikzcd}
P_{\Spin(n)}(M,g) \rar \dar & P_{\widetilde{GL}^+(n)}M  \dar \\
P_{\SO(n)}(M,g) \rar & P_{GL^+(n)}M,
\end{tikzcd}
\end{equation*}
where $P_{\SO(n)}(M,g)$ is the principal bundle of positively oriented orthonormal frames with respect to $g$. Moreover, pulling back over the canonical projection $M \times [0,1] \to M$, we obtain
\begin{equation*}
\begin{tikzcd}
{P_{\widetilde{GL}^+(n)}M \times [0,1]} \rar \dar & P_{\widetilde{GL}^+(n)}M  \dar \\
{P_{GL^+(n)}M \times [0,1]} \rar \dar & P_{GL^+(n)}M \dar \\
{M \times [0,1]} \rar & M.
\end{tikzcd}
\end{equation*}
This gives rise a topological spin structure $P_{\widetilde{GL}^+(n+1)}M \times [0,1] \to P_{GL^+(n+1)}M \times [0,1]$ on $M \times [0,1]$ by extension along the standard embedding 
\begin{align*}
GL^+(n) &\lto GL^+(n+1) \\ A &\lmapsto \begin{pmatrix} A & 0 \\ 0 & 1 \end{pmatrix}
\end{align*}
and its double covering.

Now, given a metric $g \in \Met(M)$, we can define a family of metrics by $g_t = (1-t)g_0 + tg$. Such a family in turn defines the generalized cylinder $(M \times [0,1], g_t + \upd t^2)$, $t$ being the variable in $[0,1]$-direction. As above, the topological spin structure induces a spin structure $P_{\Spin(n+1)}(M \times [0,1], g_t + \upd t^2) \to P_{\SO(n+1)}(M \times [0,1], g_t + \upd t^2)$ on the generalized cylinder. This has the property that for all $t_0 \in [0,1]$ it restricts to the spin structure of $(M,g_{t_0})$ in the sense that
\begin{equation*}
\begin{tikzcd}
P_{\Spin(n)}(M,g_{t_0}) \rar \dar & P_{\Spin(n+1)}(M \times [0,1], g_t + \upd t^2) \dar \\
P_{\SO(n)}(M,g_{t_0}) \rar & P_{\SO(n+1)}(M \times [0,1], g_t + \upd t^2)
\end{tikzcd}
\end{equation*}
is a pullback, where the lower map is the inclusion $(e_1, \ldots, e_n) \mapsto (e_1, \ldots, e_n, \Frac{\del}{\del t})$.

The reason why we do this is that on  $P_{\Spin(n+1)}(M \times [0,1], g_t + \upd t^2)$ the Levi-Civita connection induces a canonical connection $\nabla$, which provides parallel transports
\begin{align*}
P^\nabla_{\gamma_x} \colon  P_{\Spin(n+1)}(M \times [0,1], g_t + \upd t^2)_{|(x,0)} \lto  P_{\Spin(n+1)}(M \times [0,1], g_t + \upd t^2)_{|(x,1)}
\end{align*}
along the curves $\gamma_x \colon [0,1] \to M \times [0,1], \; t \mapsto (x,t)$ for all $x \in M$. These assemble into an isomorphism of principal bundles
\begin{align*}
P^\nabla \colon  P_{\Spin(n+1)}(M \times [0,1], g_t + \upd t^2)_{|M \times \{0\}} \overset{\cong}{\lto}  P_{\Spin(n+1)}(M \times [0,1], g_t + \upd t^2)_{|M \times \{1\}}.
\end{align*}
The fact that $\Frac{\del}{\del t}$ is parallel along the curves $\gamma_x$ implies that $P^\nabla$ restricts to
\begin{align*}
P^\nabla \colon P_{\Spin(n)} (M,g_0) \overset{\cong}{\lto} P_{\Spin(n)} (M,g),
\end{align*}
and this induces an isomorphism on the associated $Cl_n$-linear spinor bundles
\begin{align*}
P^\nabla \colon \Sigma_{Cl}(M,g_0) &\overset{\cong}{\lto} \Sigma_{Cl}(M,g). \\
[\tilde{\epsilon},\tilde{\phi}] &\lmapsto [P^\nabla \tilde{\epsilon}, \tilde{\phi}]
\end{align*}
Furthermore, it is immediate that $P^\nabla$ is a point-wise isometry with respect to the standard scalar products $\langle -, - \rangle$ defined on the $Cl_{n}$-linear spinor bundles.

We want to promote this to a unitary transformation between the associated $L^2$-spaces. As the $L^2$-norm also depends on the volume element, we first compare those: There exists a positive function $\beta \in \GF(M)$ such that $\dvol^g = \beta\, \dvol^{g_0}$. Then $\sqrt{\beta} P^\nabla \colon \Sigma_{Cl}(M,g_0) \to  \Sigma_{Cl}(M,g)$ induces a unitary transformation
\begin{align*}
\Phi_g \colon H \coloneqq L^2(M, \Sigma_{Cl}(M,g_0))  \overset{\cong}{\lto} L^2(M, \Sigma_{Cl}(M,g))
\end{align*}
as
\begin{align*}
(\Phi_g(\phi),\Phi_g(\psi))_{L^2} = \int_M \langle \sqrt{\beta} P^\nabla(\phi),\sqrt{\beta} P^\nabla(\psi) \rangle \, \dvol^g =\int_M \langle \phi, \psi \rangle \, \dvol^{g_0} = (\phi,\psi)_{L^2}.
\end{align*}
Moreover, it is clear that $\Phi_g$ preserves the $\Z/2\Z$-grading and the right Clifford multiplication. The left Clifford multiplication by a vector field $X \in \VF(M)$ satisfies $\Phi_g(X \cdot \phi) =P^\nabla(X) \cdot \Phi_g(\phi)$ for any $\phi \in H$, where $P^\nabla(X)$ is the vector field on $M = M \times \{1\}$ obtained from $X$ by parallel transport along the curves $(\gamma_x)_{x \in M}$ in the cylinder $(M \times [0,1], g_t + \upd t^2)$.

It is not surprising that using this identification of the $L^2$-spaces (the bounded transforms of) the Dirac operators depend continuously on the metric. For a detailed proof of the following statement see \cite[Thm.\ 2.22]{MA}.
\begin{Satz} \label{Thm:MetFred}
	The map
	\begin{align*}
	(\Met(M), \PSC(M)) &\lto (\Fred^{n,0}(H), G^{n,0}(H)) \\
	g &\lmapsto \Phi_g^{-1} \circ \frac{D_g}{\sqrt{1+D_g^2}} \circ \Phi_g
	\end{align*}
	is well-defined and continuous with respect to the $C^1$-topology on the space of smooth metrics $\Met(M)$. In particular, it is continuous if $\Met(M)$ carries the $C^\infty$-topology.
\end{Satz}

\begin{Def}
	The map from \cref{Thm:MetFred} gives rise to the composition
	\begin{align*}
	\adiff \colon \pi_k(\PSC(M),g_0) &\cong \pi_{k+1}(\Met(M),\PSC(M),g_0) \\ &\to \pi_{k+1}(\Fred^{n,0}(H),G^{n,0}(H),F_{g_0}) \overset{ind}{\lto} KO^{-n-k-1}(\{*\})
	\end{align*}
	that we call \emph{$\alpha$-index difference} or shortly \emph{$\alpha$-difference}.
\end{Def}

The $\alpha$-difference detects non-trivial homotopy groups in the space of metrics of positive scalar curvature. The following two results of this kind were independently obtained by different methods:
\begin{Satz}[Crowley, Schick, Steimle \cite{CSS}] \label{Thm:CSS}
	Let $(M,g_0)$ be a compact Riemannian spin manifold of positive scalar curvature and $n=\dim(M) \geq 6$. For all $k \geq 0$ with $k+n+1 \equiv 1,2 \mod 8$, the $\alpha$-difference
	\begin{align*}
	\adiff \colon \pi_k(\PSC(M),g_0) \lto KO^{-n-k-1}(\{*\}) \cong \Z/2\Z
	\end{align*}
	is split surjective.
\end{Satz}

\begin{Satz}[Botvinnik, Ebert, Randal-Williams \cite{BER}] \label{Thm:BER}
	Let $(M,g_0)$ be a compact Riemannian spin manifold of positive scalar curvature and $n=\dim(M) \geq 6$. For all $k \geq 0$, the $\alpha$-difference 
	\begin{align*}
		\adiff \colon \pi_k(\PSC(M),g_0) &\lto KO^{-n-k-1}(\{*\})
	\end{align*}
	is non-trivial whenever the target is non-zero, that is when $k+n+1 \equiv 0,1,2,4 \mod 8$.
\end{Satz}

We will use these results to construct non-trivial homotopy groups in the space of initial value pairs satisfying the dominant energy condition. The detection of these groups then uses an index difference for initial values that will be defined in the next chapter.

\section{An index difference for initial values}
\subsection{The  \texorpdfstring{$Cl_{n,1}$}{Cl\_\{n,1\}}-linear hypersurface spinor bundle}
Throughout this section, $(N, \overline{g})$ denotes a space- and time-oriented Lorentzian spin manifold. We follow the convention that the metric has signature $(-, +, \ldots, +)$, so that the induced metric $g$ on a spacelike hypersurface $M \subseteq N$ is positive definite. The future-pointing unit normal on $M$ will be called $e_0$. If $\overline{\nabla}$ denotes the Levi-Civita connection of $\overline{g}$ and $\nabla$ the one of $g$, the second fundamental form with respect to $e_0$ is the symmetric 2-tensor $k \in \Gamma(T^*M \otimes T^*M)$ defined by $\overline{\nabla}_X Y = k(X,Y)e_0 + \nabla_X Y$ for all $X, Y \in \VF(M)$.

We want to study the bundle obtained by restricting the $Cl_{n,1}$-linear spinor bundle of $(N, \overline{g})$ to the hypersurface $M \subseteq N$. Especially, we want to describe it intrinsically, only in terms of the pair $(g, k)$ induced on $M$. This will be of use later, when defining the $\overline{\alpha}$-difference for initial values and comparing it to the $\alpha$-difference.

The first step is to construct compatible spin structures on $M$ and $N$. Fixing a spin structure on $(N, \overline{g})$, we obtain a spin structure on $(M,g)$ by pulling back the one from $N$:
\begin{equation} \label{CD:1}
\begin{tikzcd}
P_{\Spin(n)}(M) \rar \dar & P_{\Spin_0(n,1)}(N)_{|M} \dar \\
P_{\SO(n)}(M) \rar & P_{\SO_0(n,1)}(N)_{|M}.
\end{tikzcd}
\end{equation}
Here, the lower map is given by $(e_1, \ldots , e_n)  \mapsto (e_0, e_1, \ldots, e_n)$, where $e_0$ is the future-pointing unit normal on $M$. As the right hand map is a double covering, so is the left hand one, and it suffices to construct a compatible $\Spin(n)$-action. This, we obtain by pulling back the action maps. More explicitly, there is a commutative diagram
\begin{equation} \label{CD:2}
\begin{tikzcd}
P_{\Spin(n)}(M) \times \Spin(n) \rar \dar & P_{\Spin_0(n,1)}(N)_{|M} \times \Spin_0(n,1) \dar \\
P_{\SO(n)}(M) \times \SO(n) \rar & P_{\SO_0(n,1)}(N)_{|M} \times \SO_0(n,1).
\end{tikzcd}
\end{equation}
and the desired map is the unique map from its upper-left corner to the upper-left corner of \eqref{CD:1} building, together with the other action maps, a commutative cube out of \eqref{CD:1} and \eqref{CD:2}. Note, that this commutative cube shows that $P_{\Spin(n)}(M)$ is not only a $\Spin(n)$-reduction of $P_{\SO(n)}(M)$ but also a reduction of $P_{\Spin_0(n,1)}(N)_{|M}$ with respect to the inclusion $i \colon \Spin(n) \hookrightarrow \Spin_0(n,1)$.

Next, we study associated bundles. The \emph{$Cl_{n,1}$-linear spinor bundle} 
\begin{align*}
\Sigma_{Cl}N=P_{\Spin_0(n,1)}(N) \times_\ell Cl_{n,1}
\end{align*}
is defined via the representation induced by left multiplication on $Cl_{n,1}$:
\begin{align*}
\ell \colon \Spin_0(n,1) \hookrightarrow Cl_{n,1} \lto \End(Cl_{n,1}).
\end{align*}
As noted above, $P_{\Spin(n)}(M) \to P_{\Spin_0(n,1)}(N)_{|M}$ is a $\Spin(n)$-reduction. Hence, from the theory of principal bundles (e.g. \cite[Satz 2.18]{Bau2}), it follows that
\begin{align} \label{eq:repr}
\Sigma_{Cl}N_{|M} =P_{\Spin_0(n,1)}(N)_{|M} \times_\ell Cl_{n,1} \cong P_{\Spin(n)}(M) \times_{\ell i} Cl_{n,1},
\end{align}
so the bundle $\Sigma_{Cl}N_{|M} \to M$ only depends on the Riemannian manifold $(M,g)$ and its chosen spin structure.

\begin{Def}
	The bundle $\Sigma_{Cl}N_{|M}$ from above is called \emph{$Cl_{n,1}$-linear hypersurface spinor bundle} and denoted by $\hyp M$.
\end{Def}

Similarly to the case of the $Cl_n$-linear spinor bundle, the $Cl_{n,1}$-linear hypersurface spinor bundle carries a right Clifford multiplication $R \colon \R^{n,1} \to \End(\hyp M)$ and an even-odd grading $a \colon \hyp M \to \hyp M$ as the corresponding notions for $Cl_{n,1}$ are $\Spin_0(n,1)$-invariant. Despite not being $\Spin_0(n,1)$-invariant, the scalar product $\langle - , - \rangle$ on $Cl_{n,1}$ for which the basis\footnote{For consistency with Lorentzian geometry, the basis vector of the negative definite part of $\R^{n,1}$ is called $e_0$ rather than $e_{n+1}$.} $(e_{i_1}e_{i_2}\cdots e_{i_k})_{0 \leq k \leq n,\, 0 \leq i_1 < \cdots < i_k \leq n}$ is orthonormal can be extended to $\hyp M$: Due to \eqref{eq:repr}, $\Spin(n)$-invariance of $\langle -,- \rangle$ is sufficient. This scalar product gives rise to a space of $L^2$-sections $\overline{H} \coloneqq L^2(M, \hyp M)$, on which $R$ and $a$ define a $Cl_{n,1}$-Hilbert space structure.

Yet, the trivialization of $TN_{|M}$ by $e_0$ allows us to do better. We immediately obtain the following result:

\begin{Prop} \label{Prop:Ext}
	Setting
	\begin{align*}
	\Psi \cdot e_{n+1} \coloneqq e_0 \cdot a(\Psi)
	\end{align*}
	for all $\Psi \in \hyp M$, $R$ extends to a $Cl_{n+1,1}$-multiplication
	\begin{align*}
	\tilde{R} \colon \R^{n+1,1} \to \End(\hyp M).
	\end{align*}
	that commutes with left multiplication by any $X \in TM$. Moreover, $(\overline{H}, a, \tilde{R})$ is an ample $Cl_{n+1,1}$-Hilbert space.
\end{Prop}

This $Cl_{n+1,1}$-Hilbert space structure establishes the connection to the space $H$ of $L^2$-sections of the $Cl_n$-linear spinor bundle $\Sigma_{Cl}M$.
\begin{Prop} \label{Prop:MorH}
	The $Cl_{n+1,1}$-Hilbert space $(\overline H, a, \tilde R)$ corresponds to the $Cl_n$-Hilbert space $(H,a,R)$ under the Morita equivalence described in \eqref{eq:mor1}.
	\begin{proof}
		Via this Morita equivalence, the $Cl_{n+1,1}$-Hilbert space $\overline H$  corresponds to the $Cl_{n,0}$-Hilbert space $\overline{H}_0 = \ker(\tilde{R}(e_0)\tilde{R}(e_{n+1})-\id)$ with the structure obtained by restriction. 
		
		Let us look at the endomorphism of $\R^{n,1}$ given by reflection at the hyperplane orthogonal to the line $\R e_0$.
		Viewing $\R^{n,1}$ as subset of the Clifford algebra $Cl_{n,1}$, it may be described as
		\begin{align*}
		\R^{n,1} & \lto \R^{n,1} \\
		v &\lmapsto -e_0ve_0,
		\end{align*}
		since $e_0e_0 = 1$.
		This reflection now successively induces an endomorphism: First on the Clifford algebra $Cl_{n,1}$, then by the associate bundle construction on $\hyp M$ and finally on its space of $L^2$-sections $\overline{H}$.
		The obtained endomorphism is $\tilde{R}(e_0)\tilde{R}(e_{n+1})= R(e_0)L(e_0)a$.
		We are interested in its $1$-eigenspace.
		
		On the level of $Cl_{n,1}$, the $1$-eigenspace is given by $Cl_{n} \subseteq Cl_{n,1}$, the subalgebra generated by the fixed vectors $e_1, \ldots, e_n$, whereas the $-1$-eigenspace is the complement $R(e_0)Cl_{n} \subseteq Cl_{n,1}$.
		This implies that on the level of spinor bundles
		\begin{align*}
			\hyp M \supseteq \ker(\tilde{R}(e_0)\tilde{R}(e_{n+1}) - \id) = P_{\Spin(n)}M \times_\ell Cl_{n} = \Sigma_{Cl}M
		\end{align*}
		holds. On the level of $L^2$-sections, we get
		\begin{align*}
		\overline{H}_0 = L^2(M, \ker(\tilde{R}(e_0)\tilde{R}(e_{n+1})-\id)) =L^2(M, \Sigma_{Cl} M) = H
		\end{align*}
		as required.
	\end{proof}
\end{Prop}

As a consequence of \eqref{eq:repr}, the $Cl_{n,1}$-linear hypersurface spinor bundle possesses two natural connections: On the one hand, the Levi-Civita connection $(N,\overline{g})$ induces a connection $\overline{\nabla}$ on $P_{\Spin_0(n,1)}N_{|M}$ and $\hyp M$. On the other hand, as bundle associated to $P_{\Spin(n)}M$, the bundle $\hyp M$ carries a connection $\nabla$ induced by the Levi-Civita connection of $(M,g)$. They are related by the Weingarten map (also known as shape operator):

\begin{Lem} \label{Lem:CompConn}
	For all $X \in TM$ and $\psi \in \Gamma(\hyp M)$
	\begin{align*}
	\overline{\nabla}_X \psi = \nabla_X \psi - \frac12 e_0 \cdot W(X) \cdot \psi
	\end{align*}
	holds, where $W(X)=\overline{\nabla}_X e_0$ is the Weingarten map.
	\begin{proof}
		On the tangent bundle the difference of the connections is given by $\overline{\nabla}_X Y - \nabla_X Y = k(X,Y)e_0$. As $k(X,Y)=-\overline{g}(\overline{\nabla}_X Y- \nabla_X Y,e_0)=-\overline{g}(\overline{\nabla}_X Y,e_0)=\overline{g}(Y,\overline{\nabla}_X e_0)=g(Y,W(X))$ for all $X,Y \in \VF(M)$, the Weingarten map $W$ is the endomorphism associated to the symmetric bilinear form $k$.

		In order to transfer this to the spinor bundle, let $\tilde{\epsilon}$ be a local section of $P_{\Spin(n)}M$, and $(e_1, \ldots, e_n)$ its projection to $P_{SO(n)}M$. Abusing notation, we denote by $\tilde{\epsilon}$ also its image in $P_{\Spin_0(n,1)}N_{|M}$, projecting to $(e_0,e_1, \ldots, e_n) \in P_{SO_{0}(n,1)}N_{|M}$. As the spinor bundle is associated to these spin principal bundles, we may write a spinor locally as $\psi=[\tilde{\epsilon}, \tilde\psi]$. Using the local formula for the spinorial connection (cf.~\cite[(2.5)]{BGM}), we perform the following local calculation:
		\begin{align*}
		\overline{\nabla}_X \psi - \nabla_X \psi &= [\tilde{\epsilon}, \del_X \tilde{\psi}] + \frac12 \sum_{0 \leq i < j} \epsilon_i \overline{g}(\overline{\nabla}_X e_i,e_j) e_i \cdot e_j \cdot \psi \\ 
		&\phantom{=}\,-\left([\tilde{\epsilon}, \del_X \tilde{\psi}]+ \frac12 \sum_{1 \leq i < j} \overline{g}(\nabla_X e_i,e_j) e_i \cdot e_j \cdot \psi\right) \\
		&= \frac12 \sum_{0 < j} (-1) g(\overline{\nabla}_X e_0,e_j) e_0 \cdot e_j \cdot \psi \\
		&=-\frac12  e_0 \cdot W(X) \cdot \psi,
		\end{align*}
		where $\epsilon_i = \overline{g}(e_i, e_i) \in \{\pm 1\}$.
	\end{proof}
\end{Lem}

By the way $a$, $R$ and $\langle -,- \rangle$ are defined, it is clear that they are $\nabla$-parallel. The left Clifford multiplication $L \colon TN_{|M} \otimes \hyp M \to \hyp M$ is $\nabla$-parallel as well, where $\nabla$ is defined on $TN_{|M}$ by viewing it as bundle associated to $P_{SO(n)}M$ via the lower map of \eqref{CD:1}. This can be reexpressed by saying that both the restricted left Clifford multiplication $TM \otimes \hyp M \to \hyp M$ and the endomorphism $\hyp M \to \hyp M$ given by left multiplication with $e_0$ are $\nabla$-parallel. As a consequence, the extended right Clifford multiplication $\tilde{R}$ is $\nabla$-parallel as well.

With respect to the other connection, the following can be said.  $a$, $R$ and $L$ are $\overline{\nabla}$-parallel. The scalar product $\langle -,- \rangle$, however, in general is not, as it does not originate from a $\Spin_0(n,1)$-invariant scalar product on $Cl_{n,1}$. Instead, it satisfies the following formula that follows from $\nabla$-parallelism together with \cref{Lem:CompConn}:
\begin{align*}
\del_X \langle \phi, \psi \rangle = \langle \overline{\nabla}_X \phi, \psi \rangle + \langle \phi, \overline{\nabla}_X \psi \rangle + \langle e_0 \cdot W(X) \cdot \phi , \psi \rangle.
\end{align*}

\subsection{\texorpdfstring{$Cl_{n,1}$}{Cl\_\{n,1\}}-linear Dirac-Witten operator and index difference for initial values}
\label{sec:DW-Op}
As in the previous section, let $M$ be a spacelike hypersurface of a space- and time-oriented Lorentzian spin manifold $(N,\overline{g})$. The Dirac-Witten operator is a kind of Dirac operator on the hypersurface spinor bundle. In the case of classical spinor bundles, it was first defined by Witten \cite{Wi} in order to give his spinorial proof of the positive mass theorem (cf. \cite{PT} for a rigorous formulation of the proof) and later studied in more detail by Hijazi and Zhang \cite{HZ}. We are interested in its $Cl_{n,1}$-linear version and use it to define a kind of index difference for initial values. Furthermore, we compare it to the $Cl_{n,1}$-linear Dirac operator, which will be of later use.

\begin{Def}
	The composition
	\begin{align*}
	\overline{D} \colon \Gamma(\hyp M) \overset{\overline\nabla}{\lto} \Gamma(T^*M \otimes \hyp M) &\overset{\sharp \otimes \id}{\lto} \Gamma(TM \otimes \hyp M) \overset{L}{\lto} \Gamma(\hyp M),
	\end{align*}
	where $L$ is the left Clifford multiplication, defines the \emph{$Cl_{n,1}$-linear Dirac-Witten operator}. The composition (with $\overline\nabla$ replaced by $\nabla$)
	\begin{align*}
	D \colon \Gamma(\hyp M) \overset{\nabla}{\lto} \Gamma(T^*M \otimes \hyp M) &\overset{\sharp \otimes \id}{\lto} \Gamma(TM \otimes \hyp M) \overset{L}{\lto} \Gamma(\hyp M)
	\end{align*}
	is the \emph{$Cl_{n,1}$-linear Dirac operator}.
\end{Def}

The following lemma justifies the names of these operators. It is a direct consequence of the parallelism discussion at the end of the last section.
\begin{Lem}
	$\overline{D}$ and $D$ are both $Cl_{n,1}$-linear with respect to the right Clifford multiplication $R$ and odd with respect to $a$. Furthermore, $D$ is $Cl_{n+1,1}$-linear with respect to the extended right Clifford multiplication $\tilde{R}$.
\end{Lem}

\begin{Lem} \label{Lem:DWvsD}
	$\overline{D}=D-\frac12 \tau L(e_0)$ holds, where $\tau = \tr W = \tr k$ is the mean curvature of $M$ in $N$. Both $D$ and $\overline D$ are formally self-adjoint.
	\begin{proof}
		For $\psi \in \Gamma(\hyp M)$ and a local orthonormal frame $e_1, \ldots ,e_n$ we perform the following local calculation applying \cref{Lem:CompConn}:
		\begin{align*}
		\overline D \psi - D \psi &= \sum_{i=1}^n e_i \cdot (\overline\nabla_{e_i}-\nabla_{e_i}) \psi \\
		&=- \frac12 \sum_{i=1}^n e_i \cdot e_0 \cdot W(e_i) \cdot \psi  \\
		&= \phantom{-}\frac12 \sum_{i,j=1}^n g(W(e_i),e_j) e_i \cdot e_j \cdot e_0 \cdot \psi \\
		&=-\frac12 \sum_{i=1}^n g(W(e_i),e_i) e_0 \cdot \psi.
		\end{align*}
		Here, we used that $g(W(e_i),e_j)=k(e_i,e_j)$ is symmetric in $i$ and $j$.
		
		The hypersurface spinor bundle $\hyp M$ together with the connection $\overline{\nabla}$, the (left) Clifford multiplication by $TM$ and scalar product $\langle - , - \rangle$ forms a Clifford bundle, since these structures are compatible as mentioned in the end of the last subsection.
		Since $D$ is the Dirac operator associated to this Clifford bundle, it is formally self-adjoint (cf.~\cite[Prop.~3.11]{Roe}).
		As left multiplication with $e_0$ is self-adjoint as well, the same holds true for $\overline D$.
	\end{proof}
\end{Lem}

The utility of the Dirac-Witten operator to general relativity results from following observation due to Witten \cite[eqs.\ (24)-(34)]{Wi}. The proof (cf.\ also \cite[Sec.\ 3]{PT}) verbatim applies to the $Cl_{n,1}$-linear version considered here.
\begin{Prop}
	The Dirac-Witten operator satisfies the Schrö\-ding\-er-Lich\-ne\-ro\-wicz type formula
	\begin{align*}
	\overline{D}^2 &= \overline\nabla^*\overline\nabla + \frac12 (\rho-e_0 \cdot j^\sharp \cdot), \\
	\intertext{with}
	2\rho &= \scal + \tau^2 - \|k\|^2 \\
	j &= - \upd \tau + \div k.
	\end{align*}
\end{Prop}

The Dirac-Witten operator $\overline{D}$ is elliptic, in fact it has the same principal symbol as the Dirac operator $D$.
So it possesses good functional analytic properties, some of which we will state below.
From now on, we assume that $M$ is compact.

\begin{Kor} \label{Kor:Inv}
	If the pair $(g,k)$ satisfies the strict dominant energy condition, i.e.\ if $\rho > \|j\|$, then $\overline D$ has zero kernel.
	\begin{proof}
		For any smooth section $\psi \in \Gamma(\hyp M)$ with $\psi \not\equiv 0$ 
		\begin{align*}
		\|\overline D \psi\|^2_{L^2} &=(\psi, \overline D \overline D \psi) = \|\overline\nabla \psi\|^2_{L^2} + \frac12 (\psi, \rho \psi)- \frac12 (\psi, e_0 \cdot j^\sharp \cdot \psi) \\
		& \geq \frac12 (\psi, \rho \psi)- \frac12 (\psi, \|j\| \psi) =\frac12 (\psi, (\rho - \|j\|) \psi) >0
		\end{align*}
		holds as $|\langle \psi, e_0 \cdot j^\sharp \cdot \psi \rangle| \leq \|j\| \|\psi\|^2$. Here, $\|-\|$ (without subscript $L^2$) denotes the pointwise norm.
		The claim follows, since the kernel of the elliptic differential operator $\overline{D}$ consists of smooth sections, see also \cref{Prop:EigSp} below.
	\end{proof}
\end{Kor}

\begin{Prop} \label{Prop:EigSp}
	$\overline D$ and $D$ extend to densely defined self-adjoint operators
	\begin{align*}
	D,\overline{D} \colon L^2(M,\hyp M) \supseteq H^1(M,\hyp M) \to L^2(M,\hyp M)
	\end{align*}
	admitting a spectral decomposition with discrete spectrum and finite dimensional eigen\-spaces consisting of smooth sections.
	\begin{proof}
		This is true for any formally self-adjoint elliptic differential operator of order one, for example cf.~\cite[Thm.~III.5.2 and Thm.~III.5.8]{LM}.
	\end{proof}
\end{Prop}

\begin{Kor} \label{Kor:WellDef}
	If $n=\dim(M)>0$ and $\overline H\coloneqq L^2(M,\hyp M)$, then there are well-defined elements
	\begin{align*}
	\overline{F}&\coloneqq \frac{\overline D}{\sqrt{1+{\overline D}^2}} \in \Fred^{n,1}(\overline H)
	\intertext{and}
	F&\coloneqq \frac{D}{\sqrt{1+{D}^2}} \in \Fred^{n+1,1}(\overline H) \subseteq \Fred^{n,1}(\overline H).
	\end{align*} Furthermore, $\overline F$ is invertible if $(g,k)$ satisfies the strict dominant energy condition and $F$ is invertible if $g$ has positive scalar curvature.
	\begin{proof}
		$\overline{H}$ is ample as $Cl_{n+1,1}$-Hilbert space, so it is ample as $Cl_{n,1}$-Hilbert space with the restricted Clifford action as well. As $\overline{D}$ is odd and $Cl_{n,1}$-linear, so is $\overline{F}$. From \Cref{Prop:EigSp} above, we conclude that $\overline F$ is a Fredholm operator. The additional condition in the case $n-1 \equiv -1 \mod 4$ is again a consequence of the discussion of the spectral asymptotics in the appendix.
		Invertibility for $(g,k)$ satisfying the strict dominant energy condition follows from \cref{Kor:Inv} and $\coker \overline F=\ker \overline F$. The argumentation for $F$ is completely analogous. Invertibility here uses the classical Schrödinger-Lichnerowicz formula.
	\end{proof}
\end{Kor}

If the mean curvature $\tau$ is constant, we can relate the spectral decompositions of $\overline D$ and $D$ and refine the invertibility result.
\begin{Prop} \label{Prop:EigVal}
	The spectral decomposition of $D$ can be written as
	\begin{align*}
	D = \sum_{k=0}^\infty \lambda_k \pi_{E_k} + \sum_{k=0}^\infty (-\lambda_k) \pi_{a(E_k)} 
	\end{align*}
	where all $\lambda_k >0$ are pairwise disjoint and $\pi_{E_k}$ and $\pi_{a(E_k)}$ are the orthogonal projections on the finite dimensional subspaces $E_k$ and $a(E_k)$, respectively. If the mean curvature $\tau$ is constant, then there are decompositions $F_k \oplus a(F_k)=E_k \oplus a(E_k)$ for all $k \geq 0$ and $K \oplus a(K) = \ker D$ such that the spectral decomposition of $\overline D$ is given by
	\begin{align*}
	\overline D = \sum_{k=0}^\infty \sqrt{\lambda_k^2+\frac14 \tau^2}\; \pi_{F_k} + \sum_{k=0}^\infty \left(-\sqrt{\lambda_k^2+ \frac14 \tau^2}\right) \pi_{a(F_k)} + \frac12 \tau \pi_k - \frac12 \tau \pi_{a(K)}
	\end{align*}
	In particular, $\overline D$ is invertible for all constants $\tau \neq 0$.
	\begin{proof}
		As $a$ anti-commutes with $D$, for any eigenvector $\phi$ to the eigenvalue $\lambda$
		\begin{align*}
		D a(\phi) = -a(D\phi) =-a(\lambda \phi)=-\lambda a(\phi).
		\end{align*}
		So $a(\phi)$ is an eigenvector to the eigenvalue $-\lambda$. This implies that the spectral decomposition can be written in the stated form. With the same argument, we observe that the spectral decomposition of $\overline D$ to be of that form.
		
		$\tilde R$ commutes with $D$, so the eigenspaces are invariant under $\tilde R(v)$ for all $v \in\R^{n+1,1}$. In particular,
		\begin{align*}
		a(E_k) = \tilde R(e_{n+1}) a(E_k) = L(e_0)(E_k)
		\end{align*} 
		for all $k \geq 0$. Thus we can identify $E_k$ with $a(E_k)$ via the map $E_k \to a(E_k),\; \phi \mapsto L(e_0)(\phi)$ and get $E_k \oplus a(E_k) \cong E_k \oplus E_k \cong E_k \otimes \R^2$. Under this identification, by \cref{Lem:DWvsD}, the restriction of the Dirac-Witten operator corresponds to
		\begin{align*}
		\id_{E_k} \otimes \begin{pmatrix} \lambda_k & -\frac12 \tau \\ -\frac12 \tau & -\lambda_k \end{pmatrix}.
		\end{align*}
		The characteristic polynomial of the $2 \times 2$-matrix is $x^2-\lambda_k^2-\frac14 \tau^2$, so it is diagonalizable with eigenvalues $\pm \sqrt{\lambda_k^2+\frac14 \tau^2}$. This gives rise to a diagonalization of $\overline{D}_{|E_k \oplus a E_k}$ with the same eigenvalues, and we call the positive eigenspace $F_k$.
		
		Now, we turn our attention to $\ker D$. As $L(e_0)=\tilde R(e_{n+1}) a$ anti-commutes with $D$, $L(e_0)$ operates on $\ker D$. This operation is self-adjoint and squares to $\id_{\ker D}$, so by the spectral theorem $L(e_0)_{|\ker D}$ is diagonalizable and its eigenvalues must be contained in $\{1, -1\}$. Let $K$ be the $-1$-eigenspace. Then $a(K)$ is the $1$-eigenspace. Due to 
		\begin{align*}
		\overline D_{|\ker D} = -\frac12 \tau L(e_0)_{|\ker D},
		\end{align*}
		$K$ and $a(K)$ become the $\frac12 \tau$- and $-\frac12 \tau$-eigenspaces of $\overline D$, respectively.
	\end{proof}
\end{Prop}

\begin{Bem}
	That $\overline D$ is invertible for constant mean curvature $\tau \neq 0$, can also be seen directly from the fact that $D$ anti-commutes with $L(e_0)$: As $L(e_0)^2=\id$,
	\begin{align*}
	\overline D^2 =\left(D-\frac12 \tau L(e_0)\right)^2=D^2+\frac14 \tau^2 \id
	\end{align*}
	and so $\coker \overline D=\ker \overline D=0$.
\end{Bem}

In the remainder of this section, we want to use the $Cl_{n,1}$-linear Dirac-Witten operator to define an index difference for initial values. For this, let $M$ be compact, spin and of dimension $n > 0$. We need no longer assume that it is embedded into a manifold $N$, as we succeeded in expressing all the relevant structures in terms of $M$ and the pair $(g,k)$. In fact, the $Cl_{n,1}$-linear hypersurface spinor bundle $\hyp (M,g) \cong \Sigma_{Cl}(M,g) \otimes_{Cl_n} Cl_{n,1}$
depends on the metric $g$ alone, whereas  its connection $\overline{\nabla}$ and thus its $Cl_{n,1}$-linear Dirac-Witten operator $\overline{D}$ is affected by $k$ as well. 

In analogy to the case of the $\alpha$-difference, we need to compare the spaces of $L^2$-sections of the hypersurface spinor bundles for different initial value pairs $(g,k)$. Adopting the notation from \cref{sec:alpha}, there is a bundle map
\begin{align*}
\sqrt{\beta}P^\nabla \otimes \id_{Cl_{n,1}} \colon \Sigma_{Cl}(M,g_0) \otimes_{Cl_n} Cl_{n,1} \to \Sigma_{Cl}(M,g) \otimes_{Cl_n} Cl_{n,1},
\end{align*}
which induces
\begin{align*}
\overline\Phi_g \colon \overline H \coloneqq L^2(M,\hyp (M,g_0)) &\overset{\cong}{\lto}  L^2(M,\hyp (M,g)).
\end{align*}
This allows to produce a continuous map from initial values to the space of Fredholm operators. 
\begin{Satz}[{cf.\ \cite[Thm.\ 3.19]{MA}}] \label{Thm:IniFred}
	The map
	\begin{align*}
	(\Ini(M),\DEC(M)) &\lto (\Fred^{n,1}(\overline{H}), G^{n,1}(\overline{H})) \\
	(g,k) &\lmapsto \overline\Phi_g^{-1} \circ \frac{\overline{D}_{(g,k)}}{\sqrt{1+\overline{D}_{(g,k)}^2}} \circ \overline\Phi_g
	\end{align*}
	is well-defined and continuous with respect to the $C^1$-topology on the space of smooth initial value pairs $\Ini(M)$. In particular, it is continuous if $\Ini(M)$ carries the $C^\infty$-topology.
\end{Satz}

\begin{Def}
	The \emph{$\overline{\alpha}$-difference} is defined by the composition
	\begin{align*}
	\oladiff \colon \pi_k(\DEC(M),(g_0,k_0)) &\cong \pi_{k+1}(\Ini(M),\DEC(M),(g_0,k_0)) \\ &\to \pi_{k+1}(\Fred^{n,1}(\overline H),G^{n,1}(\overline H),\overline{F}_{g_0, k_0}) \overset{\ind}{\lto} KO^{n-k}(\{*\}),
	\end{align*}
	where the middle map is the one from \cref{Thm:IniFred}.
\end{Def}

In the next chapter, $\oladiff$ will be compared to the $\alpha$-difference. The first step will be to establish a comparison map between the space of metrics of positive scalar curvature and the space of initial value pairs satisfying the dominant energy condition strictly.

\section{Comparing the index differences} 
\subsection{Positive scalar curvature and initial values} \label{sec:Susp}
In the following, $M$ is a compact smooth manifold of dimension $n \geq 2$. The aim of this section is to construct a continuous map $\Phi \colon \Susp \PSC(M) \longrightarrow \DEC(M)$, which will be used later to relate the index differences.

\begin{Lem} \label{Lem:map}
	For every $C>0$, the function
	\begin{align*}
	\tau \colon \Met(M) &\lto \R \\
	g &\lmapsto \sqrt{\frac{n}{n-1} \max\{0, \sup_{x \in M} -\scal^g(x)\}}+C
	\end{align*}
	is continuous.
\end{Lem}

\begin{Prop} \label{Prop:map1}
	For any $C>0$, the following is a well-defined continuous map of pairs:
	\begin{align*}
	\phi \colon (\Met(M),\PSC(M)) \times (I, \del I)  &\lto (\Ini(M),\DEC(M)) \\
	(g,t) &\lmapsto \left(g,\frac{\tau(g)}{n} t g \right).
	\end{align*}
	Moreover, its homotopy class $[\phi] \in [(\Met(M),\PSC(M)) \times (I, \del I)\,,\, (\Ini(M),\DEC(M))]$ is independent of $C>0$.
	\begin{proof}
		Continuity directly follows from the lemma above. Moreover, varying the parameter $C>0$ defines a continuous homotopy between different such maps. Thus, it only remains to prove that $\Met(M) \times \del I \cup \PSC(M) \times I$ is mapped into $\DEC(M)$. To this aim, we first observe that for a pair of the form $(g, \frac{\tau}{n} g)$ with $\tau \in \R$
		\begin{align*}
		2\rho &=  \scal + \frac{n-1}{n} \tau^2 \\
		j &= \frac{1-n}{n} \grad \tau = 0
		\end{align*}
		holds. Hence, such a pair fulfills the strict dominant energy condition if and only if 
		\begin{align*}
		\tau^2 > - \frac{n}{n-1} \scal.
		\end{align*}
		But by definition of the function $\tau$, this is the case for $\left(g,  \pm \frac{\tau(g)}{n} g \right)$, which shows that $\Met(M) \times \del I $ maps into $\DEC(M)$. Moreover, the condition is automatically satisfied if $g$ has positive scalar curvature, so $\PSC(M) \times I$ is sent to $\DEC(M)$ as well. 
	\end{proof}
\end{Prop}

\begin{Prop} \label{Prop:map}
	Let $C>0$ and $h \in \Met(M)$ a Riemannian metric. Then the composition
	\begin{align*}
	\Phi \colon \Susp \PSC(M) &\lto \Met(M) \times \del I \cup \PSC(M) \times I \overset{\phi}{\lto} \DEC(M), \\
	\intertext{where the first map is given by}
	[g,t] &\lmapsto
	\begin{cases}
	\left((-2t-1) h + 2(1+t)g, -1 \right) & t \in [-1,-\Frac12]\\
	(g, 2t) & t \in [-\Frac12 , \Frac12]\\
	\left((2t-1) h + 2(1-t)g, 1 \right) & t \in [\Frac12, 1],
	\end{cases}
	\end{align*}
	is a well-defined, continuous map. Its homotopy class is independent of $C>0$ and $h \in \Met(M)$.
	\begin{proof}
		By the previous proposition, we just need to study the first map: Plugging in $t=\pm \frac12$, we see that the different definitions agree on the intersections, and for the special values $t= \pm 1$ we observe that the result is independent of $g$, i.e.\ the map descends to the suspension. This shows well-definedness. Continuity can now be checked on each domain of definition, where it is obvious. Moreover, this map continuously depends on $h \in \Met(M)$, so by connectedness of $\Met(M)$, its homotopy class is independent of $h$. 
	\end{proof}
\end{Prop}

\begin{Kor}
	The inclusion $\PSC(M) \to \DEC(M)$, $g \mapsto (g,0)$ is null-homotopic. In particular, if there exists a metric $g_0 \in \PSC(M)$,
	the induced map on homotopy groups $\pi_k(\PSC(M),g_0) \to \pi_k(\DEC(M),(g_0,0))$ is the zero-map for all $k$.
	\begin{proof}
		Using the map defined above, we get a factorization of the inclusion map as follows
		\begin{align*}
		\PSC(M) \hookrightarrow C \PSC(M) \hookrightarrow \Susp \PSC(M) \overset{\Phi}{\lto} \DEC(M),
		\end{align*}
		where the first two maps are the canonical inclusions of a space into the its cone and of the cone into the suspension as upper half. As cones are contractible, the composition is null-homotopic.
	\end{proof}
\end{Kor}

This shows that we cannot find non-trivial elements of homotopy groups in the space initial data with strict dominant energy condition by simply considering the space of positive scalar curvature metrics as subspace. However, the map $\Phi$ defined above allows for a better construction: In the remaining section, we will show that under certain conditions the composition
\begin{align*}
\pi_k(\PSC(M),g_0) \overset{\mathrm{Susp}}{\lto} \pi_{k+1}(\Susp \PSC(M), [g_0,0]) \overset{\Phi_*}{\lto} \pi_{k+1}(\DEC(M),(g_0,0))
\end{align*}
has non-trivial image.

\subsection{Main theorem}
Let $M$ be a compact spin manifold of dimension $n \geq 2$. The aim of this section is to relate the $\overline{\alpha}$-difference for initial values $\oladiff \colon \pi_k(\DEC(M), (g_0,0)) \to KO^{-n-k}(\{*\})$, where $g_0$ is a metric of positive scalar curvature, to the classical $\alpha$-difference using the map from \cref{Prop:map}. This will lead to a non-triviality result for $\pi_k(\DEC(M),(g_0,0))$. Moreover, the same argument shows that the $\overline{\alpha}$-difference detects that $\DEC(M)$ has least two connected components if $\alpha(M) \neq 0$.

\begin{Satz}[Main Theorem] \label{MainThm}
\begin{enumerate}
	\item 	If $M$ carries a metric $g_0$ of positive scalar curvature, then for all $k \geq 0$, the diagram
	\begin{equation*}
	\begin{tikzcd}
	\pi_k(\PSC(M), g_0)  \ar[dr, "\adiff"'] \rar{\mathrm{Susp}} & \pi_{k+1}(\Susp\PSC(M), [g_0,0]) \rar{\Phi_*}&
	\pi_{k+1}(\DEC(M), (g_0,0)) \ar{dl}{\oladiff} \\ 
	&  KO^{-n-k-1}(\{*\}) &
	\end{tikzcd}
	\end{equation*}
	commutes. Here, $\mathrm{Susp}$ is the suspension homomorphism and $\Phi$ is the map from \cref{Prop:map}.
	\item	For any metric $g_0$,
	\begin{align*}
	\oladiff\left(\left(g_0,-\frac{1}{n}\tau(g_0)g_0 \right), \left(g_0,\frac{1}{n}\tau(g_0)g_0 \right)\right) = \alpha(M) \in KO^{-n}(\{*\}),
	\end{align*}
	where $\tau$ is defined as in \cref{Lem:map}.
\end{enumerate}
\end{Satz}
\begin{proof}
	For the first part, we start by exploring the effect of the upper composition. The claim is that
	\begin{equation}
	\begin{tikzcd}[column sep=-0.75cm] \label{eq:diag0}
	\pi_k(\PSC(M),g_0) \rar{\mathrm{Susp}}& \pi_{k+1}(\Susp\PSC(M), [g_0,0])  \rar{\Phi_*} &\pi_{k+1}(\DEC(M),(g_0,0)) \\
	\pi_{k+1}(\Met(M),\PSC(M),g_0)\dar \uar{\cong} \ar[rr, "\phi_*"] & & \pi_{k+2}(\Ini(M),\DEC(M),(g_0,0)) \dar \uar{\cong} \\
	{[(D^{k+1},S^k),(\Met(M),\PSC(M))]}\ar[rr, "\phi_*"] & &{[(D^{k+1},S^k)\times(I,\del I),(\Ini(M),\DEC(M))]}
	\end{tikzcd}
	\end{equation}
	commutes, where the middle and the lower map are both induced by
	\begin{align*}
	\phi \colon (\Met(M),\PSC(M))\times (I,\del I) &\lto (\Ini(M),\DEC(M)) \\
	(g,t) &\lmapsto \left(g, \frac{\tau(g)}{n}tg \right).
	\end{align*}
	Note that $\phi$ preserves the base point, if the base point of $(D^{k+1},S^k)\times(I,\del I)$ is chosen to be $(*,0)$ when $*$ is the base point of $S^k$, so the middle map is well-defined. The lower square obviously commutes. For the upper square, we start with a class $[g] \in \pi_k(\PSC(M),g_0)$. Then the preimage under the boundary isomorphism is represented by
	\begin{align*}
	\tilde{g} \colon (D^{k+1}, S^k,*) &\lto (\Met(M),\PSC(M),g_0) \\
	rx &\lmapsto (1-r)g_0+rg(x)
	\end{align*}
	for $r \in [0,1]$ and $x \in S^k$. Applying the horizontal map and restricting to the boundary yields the class of
	\begin{align*}
	(\del(D^{k+1} \times I),(*,0)) &\lto (\DEC(M),(g_0,0)) \\
	(x,t) &\lmapsto \left(\tilde{g}(x), -\frac{\tau(\tilde{g}(x))}{n}t\tilde{g}(x) \right).
	\end{align*}
	Using the homeomorphism
	\begin{align*}
	(\Susp(S^k),[*,0]) &\cong (\del(D^{k+1} \times I),(*,0))\\
	[x,t]&\mapsto \begin{cases}
	(2(1+t)x,-1) &  t \in [-1,-\Frac12]\\
	(x,2t) & t \in [-\Frac12 , \Frac12] \\
	(2(1-t)x,1) & t \in [\Frac12, 1],
	\end{cases} 
	\end{align*}
	this precisely gives the formula for $\Phi \circ \Susp g$ (cf.\ \cref{Prop:map}).
	
	The core of the proof is showing that the following diagram commutes:
	\begin{equation} \label{eq:diag1}
	\begin{tikzcd}[column sep=-3.2cm]
	{[(D^{k+1},S^k),(\Met(M),\PSC(M))]} \dar \ar[rr, "\phi_*"] & &{[(D^{k+1},S^k)\!\times\!(I,\del I),(\Ini(M),\DEC(M))]} \dar \\[0.9em]
	{[(D^{k+1},S^k),(\Fred^{n,0}(H),G^{n,0}(H))]} \drar{\cong} &  &  {[(D^{k+1},S^{k})\!\times\!(I,\del I),(\Fred^{n,1}(\overline H),G^{n,1}(\overline H))]} \\
	& {[(D^{k+1},S^k),(\Fred^{n+1,1}(\overline H),G^{n+1,1}(\overline H))].} \urar{\cong} & 
	\end{tikzcd}
	\end{equation}
	Here, the first lower map is associated to the Morita equivalence between $Cl_{n,0}$- and $Cl_{n+1,1}$-Hilbert spaces, that is the first map in \cref{Prop:MorMap}. This uses that $H$ and $\overline{H}$ correspond to each other under this Morita equivalence according to \cref{Prop:MorH}. The second lower map is the Bott map (cf.~\cref{Thm:KOviaCl2}), associated to $e=-e_{n+1}$.
	
	Before doing so, let us show that
	\begin{equation} \label{eq:diag2}
	\begin{tikzcd}[column sep=-3.2cm, row sep=+1cm]
	& &[-1.2cm]  {[(D^{k+1},S^{k})\!\times\!(I,\del I),(\Fred^{n,1}(\overline H),G^{n,1}(\overline H))]}\ar[dd, shift right=-6ex, "\cong"]\\
	& {[(D^{k+1},S^k),(\Fred^{n+1,1}(\overline H),G^{n+1,1}(\overline H))]} \urar{\cong} \drar{\cong} &  \\
	{[(D^{k+1},S^k),(\Fred^{n,0}(H),G^{n,0}(H))]} \urar{\cong} \ar[dd, shift left=-3ex,"\cong"' near start, "\ind"'] \drar{\cong}&  & {[(D^{k+1},S^{k})\!\times\!(I,\del I),(\Fred^{n,1}(\overline H),G^{n,1}(\overline H))]}\ar[dd, shift right=-6ex, "\cong"  near start, "\ind"] \\
	& {[(D^{k+1},S^k)\!\times\!(I,\del I),(\Fred^{n-1,0}(H),G^{n-1,0}(H))]} \urar{\cong} \ar[rd,"\cong"', "\ind"] &  \\
	KO^{-n}(D^{k+1},S^k) \drar[swap]{\cong} \ar[rr,"\cong"]& & KO^{-n+1}((D^{k+1},S^{k})\!\times\!(I,\del I)) \dlar{\cong} \\
	& KO^{-n-k-1}(\{*\}) &
	\end{tikzcd}
	\end{equation}
	commutes. Here the central diamond is formed by the Bott maps associated to $e=e_n$ as well as maps induced by Morita equivalences. The topmost right hand map is induced by a $Cl_{n,1}$-Hilbert space isomorphism to be defined later. Notice that the right hand vertical composition is the index map, which follows from the invariance of the index map under $Cl_{n,1}$-Hilbert space isomorphisms. So stitching the diagrams \eqref{eq:diag0}-\eqref{eq:diag2} together, we obtain the diagram from the first claim.
	
	Moreover, setting $k = -1$, the commutative diagram composed of \eqref{eq:diag1} and \eqref{eq:diag2} implies the second assertion. Then $(D^{k+1}, S^k) = (\{*\},\emptyset)$ and the upper left corner of the diagram is the one-point set $[\{*\},\Met(M)]$. Now the left hand vertical composition maps this point to the $\alpha$-index of $M$, whereas the composition through the upper right corner is seen to map it to the $\overline{\alpha}$-difference of the $\pi_0$-class from the claim.
	
	The lower half of \eqref{eq:diag2} commutes by the definition of the index map, cf.\ \eqref{eq:indbott} and \eqref{eq:indmor}. The middle diamond commutes as well, this is obvious from the way its constituting maps are defined. We are left with the upper triangle. Note first that we are dealing with two different $Cl_{n,1}$-Hilbert space structures on $\overline{H}$: Since the map from the center upwards is the Bott map for $e=-e_{n+1}$, the $Cl_{n,1}$-structure is the one obtained by forgetting the $\tilde{R}(e_{n+1})$-action, whereas in the lower Hilbert space, we forget the multiplication by $e_n$. These are connected by the $Cl_{n,1}$-Hilbert space isomorphism
	\begin{align*}
	U \colon \overline{H} &\lto \overline{H} \\
	\phi &\mapsto \frac{1}{\sqrt{2}} \tilde{R}(e_{n+1})\tilde{R}(e_n+e_{n+1}).
	\end{align*}
	Indeed, $a \in B(\overline H)$ corresponds via $U$ to $a=U a U^{-1}$, $\tilde{R}(e_i)$ to $\tilde{R}(e_i)$ for $i<n$ and $\tilde{R}(e_n)$ to $\tilde{R}(e_{n+1})$. The right hand map in the triangle is defined to be the map induced by $\Fred^{n,1}(\overline H) \ni F \mapsto UFU^{-1}$. As the analogous map on $\Fred^{n+1,1}(\overline H)$ is the identity, the diagram relating the Bott maps gets the shape of a triangle rather than a square. Its commutativity follows from
	\begin{align*}
	U\tilde{R}(-e_{n+1})U^{-1} &= \frac12 \tilde{R}(e_{n+1})\tilde{R}(e_n+e_{n+1})\tilde{R}(-e_{n+1})\tilde{R}(e_n+e_{n+1})\tilde{R}(e_{n+1}) \\
	&=\frac12 (\tilde{R}(e_{n+1})+\tilde{R}(e_n)+\tilde{R}(e_n)-\tilde{R}(e_{n+1})) =\tilde{R}(e_n).
	\end{align*}
	
	It only remains prove that \eqref{eq:diag1} commutes. The first two maps of the lower composition map $[g] \in [(D^{k+1},S^k),(\Met(M),\PSC(M))]$ to the class of
	\begin{align*}
	(D^{k+1},S^k)&\lto (\Fred^{n+1,1}(\overline H),G^{n+1,1}(\overline H)) \\
	x &\lmapsto  \Phi^{-1}_{g(x)} \frac{D_{g(x)}}{\sqrt{1+D_{g(x)}^2}} \Phi_{g(x)}.
	\end{align*}
	This is because it restricts to the correct map on $H=\ker(\tilde{R}(e_0)\tilde{R}(e_{n+1})-\id) \subseteq \overline{H}$, i.e.\ the $Cl_n$-Hilbert space associated to $\overline{H}$ via the Morita equivalence \eqref{eq:mor1}. The remaining map sends it to the class of
	\begin{align*}
	(D^{k+1},S^k)\times(I,\del I) &\lto (\Fred^{n,1}(\overline H),G^{n,1}(\overline H)) \\
	(x,t) &\lmapsto \Phi^{-1}_{g(x)} \frac{D_{g(x)}}{\sqrt{1+D_{g(x)}^2}} \Phi_{g(x)}-t\tilde{R}(e_{n+1}) a  \\
	&\phantom{\lmapsto} = \Phi^{-1}_{g(x)} \left(\frac{D_{g(x)}}{\sqrt{1+D_{g(x)}^2}}-tL(e_0) \right)\Phi_{g(x)}.
	\end{align*}
	In contrast, the result of the upper composition is represented by
	\begin{align*}
	(D^{k+1},S^k)\times(I,\del I) &\lto (\Fred^{n,1}(\overline H),G^{n,1}(\overline H)) \\
	(x,t) &\lmapsto \Phi^{-1}_{g(x)} \frac{\overline{D}_{(g(x),k(x,t))}}{\sqrt{1+\overline{D}_{(g(x),k(x,t))}^2}} \Phi_{g(x)} 
	\end{align*}\enlargethispage{2\baselineskip}
	with $k(x,t)=\frac{\tau(g(x))}{n}tg(x)$.
	
	Remembering that $\overline{D}_{(g,k)}=D_g-\frac12 \tau L(e_0)$, these do not look too much different, and we show that the following is a well-defined homotopy between them:
	\begin{align*}
	(D^{k+1},S^k)\!\times\!(I,\del I)\!\times\![0,1] &\to (\Fred^{n,1}(\overline H),G^{n,1}(\overline H)) \\
	(x,t,s) &\mapsto \Phi^{-1}_{g(x)} \left(a_{(x,t,s)}(D_{g(x)}) D_{g(x)}
	- b_{(x,t,s)}(D_{g(x)})tL(e_0) \right)\Phi_{g(x)}
	\end{align*}
	for
	\begin{align*}
	a_{(x,t,s)}(\lambda)&=\frac{s}{\sqrt{1+\lambda^2}}+\frac{1-s}{\sqrt{1+\lambda^2+\frac14 t^2\tau(g(x))}} \\
	b_{(x,t,s)}(\lambda)&=s+\frac{(1-s)\frac12 \tau(g(x))}{\sqrt{1+\lambda^2+\frac14 t^2\tau(g(x))}}.	
	\end{align*}
	As this operator family is obtained by linearly interpolating between two continuous operator families, it is again continuous. So it remains to see that its target is indeed $(\Fred^{n,1}(\overline H),G^{n,1}(\overline H))$. 
	It is clear, that all the operators are bounded, self-adjoint, odd and $Cl_{n,1}$-linear. To show that the operator $F_{(x,t,s)}$ associated to $(x,t,s)$ is Fredholm, we use the spectral decomposition of $D_{g(x)}$ from \cref{Prop:EigVal}: The restriction of $F_{(x,t,s)}$ to $E_k \oplus a(E_k) \cong E_k \otimes \R^2$ is given by
	\begin{align*}
	\id_{E_k} \otimes \begin{pmatrix} a_{(x,t,s)}(\lambda_k)\lambda_k & -b_{(x,t,s)}(\lambda_k)t \\ -b_{(x,t,s)}(\lambda_k)t & -a_{(x,t,s)}(\lambda_k)\lambda_k \end{pmatrix}.
	\end{align*}
	This is diagonalizable with eigenvalues $\pm \sqrt{a_{(x,t,s)}(\lambda_k)^2\lambda_k^2+b_{(x,t,s)}(\lambda_k)^2t^2}$. Note that due to $\sqrt{a_{(x,t,s)}(\lambda_k)^2\lambda_k^2+b_{(x,t,s)}(\lambda_k)^2t^2} \geq a_{(x,t,s)}(\lambda_k) |\lambda_k|$, their absolute values, for any $t \in I$ and $s \in [0,1]$, are bounded away from zero by
	\begin{align*}
	\frac{\lambda_0}{\sqrt{1+\lambda_0^2+\frac14 \tau(g(x))}} >0,
	\end{align*}
	where $\lambda_0>0$ denotes the smallest positive eigenvalue of $D_{g(x)}$. A similar consideration as in \cref{Prop:EigVal} shows that $F_{(x,t,s)}$ restricted to $\ker(D_{g(x)})$  is diagonalizable as well, with eigenvalues $\pm b_{(x,t,s)}(0)t$. Putting this together, we find that $F_{(x,t,s)}$  has finite dimensional kernel, co-kernel and closed image (for this, the boundedness away from zero is needed). Furthermore, $F_{(x,t,s)}$ is invertible if $D_{g(x)}$ is invertible or $t >0$, one of which is the case on $\del (D^{k+1}\times I)$.
	
	In the case $n-1 \equiv -1 \mod 4$ one more tiny bit of thought is necessary. The space self-adjoint $Cl_{n,1}$-linear Fredholm operators has three components (cf.\ \cite{AS}): Those $F$ for which $\omega_{n,1}F\iota$ is essentially positive, those for which it is essentially negative and the rest. As for $s=0$ (or $s=1$) all operators $F_{(x,t,s)}$ fall into the last category, the same has to be true for all $s \in [0,1]$ by continuity. 
\end{proof}

\subsection{Corollaries and examples}
In this final section, we explore some of the consequences of the main theorem (\cref{MainThm}).
We start by combining the first part of the main theorem with the non-triviality results for the $\alpha$-difference from \cref{Thm:CSS,Thm:BER}.
This gives the following conclusions:
\begin{Kor}
	If $M$ is a closed spin manifold of dimension $n \geq 6$ that carries a metric $g_0$ of positive scalar curvature, then for all $k \geq 1$ with $k+n \equiv 1,2 \mod 8$ the $\overline{\alpha}$-difference for initial values $\oladiff \colon \pi_k(\DEC(M),(g_0,0)) \to  KO^{-n-k}(\{*\}) \cong \Z/2\Z$ is split surjective.
\end{Kor} 
\begin{Kor}
	If $M$ is a closed spin manifold of dimension $n \geq 6$ that carries a metric $g_0$ of positive scalar curvature, then for all $k \geq 1$ the $\overline{\alpha}$-difference for initial values $\oladiff \colon \pi_k(\DEC(M),(g_0,0)) \to  KO^{-n-k}(\{*\})$ is non-trivial whenever the target is non-zero, that is when $k+n \equiv 0,1,2,4 \mod 8$.
\end{Kor}
In particular, under the assumptions of the corollaries above, $\pi_k(\DEC(M),(g_0,0)) \neq 0$, which shows the first part of \cref{Kor:Main}. Note that the main theorem provides an explicit construction of the non-trivial elements, provided that in $\pi_{k-1}(\PSC(M),g_0)$ the non-trivial elements detected by the $\alpha$-difference are known.

Particularly much is known about connected components of the space of positive scalar curvature metrics.
If there are several components of $\PSC(M)$ that can be distinguished by their $\alpha$-index difference, the main theorem provides us with non-trivial loops in $\DEC(M)$.

\begin{Bsp}
	As explained in \cite[Ex.~IV.7.5]{LM}, there is a sequence of positive scalar curvature metrics $g_k \in \PSC(S^7)$, $k \in \Z$, on the (standard) $7$-sphere with the following property:
	If $V_k \to S^4$ is the real vector bundle with Euler number $\chi = 1$ and Pontrjagin number $p_1 = 4 + 896k$, then, after identifying its sphere bundle $\del D(V_k)$ with $S^7$, the metric $g_k$ extends to a positive scalar curvature metric $\hat{g}_k$ on the disk bundle $D(V_k)$ collared along the boundary.
	All these metrics $g_k$ lie in different path components of $\PSC(S^7)$.
	More precisely, $\adiff(g_k, g_l) = l - k$.
	This can be seen as follows:
	According to the main result of \cite{Eb}, $\adiff(g_k, g_l)$ is equal to the index of the $Cl_8$-linear Dirac operator on $S^7 \times \R$ equipped with a metric of the form $\hat{h} = h_t + \upd t^2$, where $h_t = g_k$ for $t \leq -1$ and $h_t = g_l$ for $t \geq 1$.
	Under complexification and Bott periodicity $KO^{-8}(\{*\}) \cong K^{-8}(\{*\}) \cong K^0(\{*\}) \cong \Z$, this corresponds to the index of the classical Dirac operator on $(S^7 \times \R, \hat{h})$.
	We compute this using the cut-and-paste version version of relative index theorem (cf.~\cite[Thm.~1.2]{Bu}).
	We take the double of $(D(V_k), \hat{g}_k)$ and cut it along the former boundary $\del D(V_k)$.
	We also cut $(S^7 \times \R, \hat{h})$ along $S^7 \times \{-1\}$.
	Then, using the identification $S^7 \cong \del D(V_k)$, we glue them together in the other way that respects the boundary orientations.
	For the indices of the associated Dirac operators, we obtain:
	\begin{align*}
			\mathrm{index}(S^7 \times \R, \hat{h}) = 
				&-\mathrm{index}(D(V_k) \cup (-D(V_k)), \hat{g}_k \cup \hat{g}_k) \\
				&+ \mathrm{index}((S^7 \times (-\infty, -1])  \cup (-D(V_k)), \hat{h} \cup \hat{g}_k) \\
				&+ \mathrm{index}(D(V_k) \cup (S^7 \times [-1, \infty)), \hat{g}_k \cup \hat{h}).
	\end{align*}
	Here, the two first indices vanish since the metric has positive scalar curvature.
	Proceeding similarly at $S^7 \times \{1\}$, we get
	\begin{align*}
			\mathrm{index}(S^7 \times \R, \hat{h}) &= \mathrm{index}(D(V_k) \cup (S^7 \times [-1, \infty)), \hat{g}_k \cup \hat{h}) \\
					&= \mathrm{index}(D(V_k) \cup (S^7 \times [-1, 1]) \cup (-D(V_l)), \hat{g}_k \cup \hat{h} \cup \hat{g}_l).
	\end{align*}
	The latter is of course equal to $\hat{A}(D(V_k) \cup (-D(V_l)))$.
	Using cut-and-paste once more, the claimed equality with $k-l$ reduces to the statement $\hat{A}(D(V_k) \cup D^8) = k$ from \cite[Ex.~IV.7.5]{LM}.
	
	Now, the suspension construction from \cref{sec:Susp} produces an element in $\pi_1(\DEC(S^7))$ out of the $\pi_0$-class defined by $g_k$ and $g_l$.
	If $k \neq l$, the main theorem shows that its $\overline{\alpha}$-difference is $k - l \neq 0$, hence it is non-trivial.
	Tracking through the definitions, it is represented by a loop that is concatenated from the following four segments:
	In the first segment the initial value pairs are all of the form $(g, \frac{1}{n} \tau(g)g)$ and the metric $g$ interpolates between $g_l$ and $g_k$.
	The second segment is a linear interpolation between $(g_k, \frac{1}{n} \tau(g_k)g_k)$ and $(g_k, -\frac{1}{n} \tau(g_k)g_k)$.
	In particular, the first component of the initial value pair is fixed throughout the second segment.
	The third piece consists of initial value pairs $(g, -\frac{1}{n} \tau(g)g)$, where $g$ runs from $g_k$ to $g_l$.
	The final segment is again an interpolation within the second component only, running from $(g_l, -\frac{1}{n} \tau(g_l)g_l)$ to $(g_l, \frac{1}{n} \tau(g_l)g_l)$.
	We have thus found a rather explicit infinite family of non-trivial elements in $\pi_1(\DEC(S^7))$.
\end{Bsp}

Concerning path components of $\DEC(M)$, we can say the following.
It is easy to see that all pairs $(g, \frac{1}{n}\tau(g)g)$, $g \in \Met(M)$, lie in the same path component of $\DEC(M)$.
The same is true for all pairs of the form $(g, -\frac{1}{n}\tau(g)g)$.
If $M$ carries a positive scalar curvature metric, then the components of $(g, \frac{1}{n}\tau(g)g)$ and $(g, -\frac{1}{n}\tau(g)g)$ are actually the same.
If on the other hand $\alpha(M) \neq 0$ (and hence $M$ does not admit positive scalar curvature), the second part of the main theorem shows that these belong to different path component as their $\overline{\alpha}$-difference is non-zero.
This immediately implies the second part of \cref{Kor:Main}.
It is the purpose of the follow-up work \cite{Gl} to show that we can still distinguish these two path components if $M$ does not carry a positive scalar curvature metric due to the (also index-theoretic) enlargeability obstruction.
In special cases, we may be able to distinguish more components.

\begin{Bsp}
	Consider the connected sum $M = K3 \# K3$, which we decompose into $(K3 \setminus D^4) \cup  
	(S^3 \times [-L,L]) \cup (K3 \setminus D^4)$, $L > 0$.
	Choose a metric $g$ on $M$ that is symmetric under the involution $\sigma$ switching the two $K3$-surfaces and reflecting the $[-L,L]$-component of the connecting neck.
	We assume moreover that $g$ is the standard product metric on the neck $S^3 \times [-L,L]$.
	Observe that we can make the neck longer, i.~e.\ $L$ larger, without changing $\tau(g)$.
	Since $\alpha(M) = 2\alpha(K3) \neq 0$, we already know that $(g, -\frac{1}{n}\tau(g)g)$ and $(g, \frac{1}{n}\tau(g)g)$ lie in different path components.

	We now consider the following initial value pair $(g,k)$.
	On the left $K3 \setminus D^4$, it is given by $(g, -\frac{1}{n} \tau(g)g)$.
	On the right $K3 \setminus D^4$, it is $(g, \frac{1}{n} \tau(g)g)$.
	Along the neck, we take $(g, \frac{t}{nL} \tau(g)g)$ at $(x,t) \in S^3 \times [-L,L]$.
	By the definition of $\tau$, the so obtained initial value pair satisfies the strict dominant energy condition along the two $K3$-parts.
	Since the metric on $S^3 \times [-L, L]$ has positive scalar curvature, the estimate
	\begin{align*}
		\rho - \|j\| &\geq \scal^g - \frac{n-1}{nL} \tau(g)
	\end{align*}
	shows that this initial value pair also satisfies the strict dominant energy condition in the neck region as long as $L$ is chosen to be large enough.
	Thus we have constructed an element $(g,k) \in \DEC(M)$ and we claim that it is part of neither of two components mentioned before.
	Assume that there were a path $t \mapsto (g_t,k_t)$ in $\DEC(M)$ connecting $(g,k)$ to $(g, \frac{1}{n}\tau(g)g)$, say.
	Then $(\sigma^* g_t, -\sigma^* k_t)$ would be a path in $\DEC(M)$ connecting it also with $(g, -\frac{1}{n}\tau(g)g)$, contradiction.
	
	It might be worth noting that the component of the pair $(g,k)$ constructed above may be detected by the $\overline{\alpha}$-difference.
	Namely, it is not hard to see that it is additive in the sense
	\begin{align*}
		\oladiff\left(\left(g,-\frac{1}{n}\tau(g)g\right), (g,k)\right) &+ \oladiff\left((g,k), \left(g,\frac{1}{n}\tau(g)g\right)\right) \\
		 &= \oladiff\left(\left(g,-\frac{1}{n}\tau(g)g\right), \left(g,\frac{1}{n}\tau(g)g\right)\right) = 2\alpha(K3).
	\end{align*}
	Moreover, replacing the endomorphism $L(e_0)$ by $-L(e_0)$ the Dirac-Witten operators defining $\oladiff((g,-\frac{1}{n}\tau(g)g), (g, -k))$ turn on the nose into the Dirac-Witten operators defining $\oladiff((g,\frac{1}{n}\tau(g)g), (g, k))$.
	Hence,
	\begin{align*}
		\oladiff\left((g,k), \left(g,\frac{1}{n}\tau(g)g\right)\right)
			&= -\oladiff\left(\left(g,\frac{1}{n}\tau(g)g\right), (g,k)\right) \\
			&= \oladiff\left(\left(g,-\frac{1}{n}\tau(g)g\right), (g, -k)\right) \\
			&= \oladiff\left(\left(g,-\frac{1}{n}\tau(g)g\right), (g, k)\right),
	\end{align*}
	where the last step uses the invariance of the $\overline{\alpha}$-difference under the diffeomorphism $\sigma$.
	We obtain
	\begin{align*}
		\oladiff\left((g,k), \left(g,\frac{1}{n}\tau(g)g\right)\right) = \oladiff\left(\left(g,-\frac{1}{n}\tau(g)g\right), (g, k)\right) &= \alpha(K3) \neq 0.
	\end{align*}
	This result can probably also be obtained with the help of a suitable relative index theorem.
\end{Bsp}

\appendix
\section{On the spectral asymptotics}
This appendix is devoted to the fact that the spectrum of a formally self-adjoint, first order elliptic differential operator has both infinitely many positive and infinitely many negative eigenvalues.
This is used in the text when the operator is $\omega_{n,0}D\iota$ for $n \equiv -1 \mod 4$ or $\omega_{n,1}\overline{D}\iota$ for $n \equiv 0 \mod 4$, where $n > 0$ is as always the dimension of the manifold.
Although this statement is probably well-known, it is hard to find a reference in the literature.
The following argument was suggested by the anonymous referee.

\begin{Prop}
Let $M$ be a closed Riemannian manifold of dimension $n \geq 1$ and $E \to M$ be a vector bundle with a metric and a metric connection $\nabla$.
Assume that $D \colon \Gamma(E) \to \Gamma(E)$ is a formally self-adjoint, first order elliptic differential operator.
Then $D$ has infinitely many positive and infinitely many negative eigenvalues.
\begin{proof}	
First of all, after potentially passing to the complexification, we may assume that $E \to M$ is a complex vector bundle.
Note that the assumptions on $D$ together with the compactness of $M$ guarantee that the spectrum of $D$ is discrete and consists of real eigenvalues with finite multiplicity (cf.~\cite[Thm.~III.5.8]{LM}).
We assume for contradiction that the spectrum is bounded below.
Then, replacing $D$ by $D + c$ for some $c \in \R$, we may assume that $D$ is positive.

Now take a covector $\xi \in T^*_pM$ so that $\sigma_D(\xi) \neq 0$, where the principal symbol $\sigma_D$ of $D$ is defined by $\sigma_D(\upd f) = [D, f]$ for any $f \in C^\infty(M)$.
In fact, since $D$ is elliptic, any $\xi \neq 0$ will do the job.
Since endomorphism  $i\sigma_D(\xi)$ is self-adjoint, we may choose an eigenvector $\Psi_p \in E_p$ of non-zero eigenvalue.
Let $f \in C^\infty(M)$ be a function with $\upd_p f = \xi$ and $\Psi \in \Gamma(E)$ be a section extending $\Psi_p$.
Since $\langle \Psi_p, i\sigma_D(\xi) \Psi_p \rangle \neq 0$,  we will have $(\Psi, i\sigma_D(\upd f) \Psi)_{L^2} \neq 0$ -- at least after multiplying $\Psi$ with a cut-off function supported near $p$.

For any $t \in \R$, we have
\begin{align*}
	e^{-itf}D(e^{itf}\Psi) &= D \Psi + e^{itf}[D, e^{itf}] \Psi = D \Psi + e^{-itf}\sigma_D(\upd e^{itf}) \Psi = D \Psi + it\sigma_D(\upd f) \Psi
\end{align*}
and thus
\begin{align*}
	(e^{itf}\Psi, D e^{itf}\Psi)_{L^2} &= (\Psi, e^{-itf} D e^{itf}\Psi)_{L^2} = (\Psi, D \Psi)_{L^2} + t(\Psi, i\sigma_D(\upd f) \Psi)_{L^2}.
\end{align*}
This yields the desired contradiction since positivity of $D$ implies that the left-hand side $(e^{itf}\Psi, D e^{itf}\Psi)_{L^2} \geq 0$ for every $t \in \R$.
\end{proof}
\end{Prop}

\section*{Declarations}
The author declares that there is no conflict of interest. There is no associated data or code.

\addcontentsline{toc}{section}{References}
\emergencystretch=1em
\printbibliography
 
\end{document}